\documentclass[12pt]{amsart}

\usepackage[dvipsnames]{xcolor}
\usepackage{pgf,tikz}
\usetikzlibrary{arrows}

\usepackage{amssymb}
\usepackage{latexsym}
\usepackage{enumerate}
\usepackage{amscd}

\usepackage{tikz}
\usepackage[all]{xy}
\usepackage{graphics}
\usepackage[active]{srcltx}
\usepackage{mathtools}

\DeclareMathOperator{\tr}{tr}
\newcommand{\HH}{\mathrm{H}}
\newcommand{\HF}{\operatorname{H\underline{\F}}}
\newcommand{\Z}{\mathbf{Z}}

\newcommand{\C}{\mathbf{C}}
\newcommand{\F}{\mathbf{F}}

\newcommand{\CC}{\mathbb{C}}

\newcommand{\ZT}{\C_2 \mathcal{T}}
\newcommand{\ZS}{\C_2\mathcal{S}p}
\newcommand{\ZhS}{\C_2h\mathcal{S}p}
\newcommand{\Sp}{\mathcal{S}p}
\newcommand{\T}{\mathcal{T}}

\newcommand{\bF}{\mathbf{F}}

\renewcommand{\mod}{\mathrm{mod}}
\renewcommand{\hom}{\mathrm{Hom}}

\numberwithin{equation}{section}

\newtheorem{theorem}[equation]{Theorem}
\newtheorem{proposition}[equation]{Proposition}
\newtheorem{corollary}[equation]{Corollary}
\newtheorem{lemma}[equation]{Lemma}

\theoremstyle{definition}
\newtheorem{example}[equation]{Example}
\newtheorem{definition}[equation]{Definition}
\newtheorem{remark}[equation]{Remark}
\newtheorem{convention}[equation]{Convention}

\newcommand{\wmono}{ \ar@{>->}[r]}
\newcommand{\wmonovert}{ \ar@{>->}[d]}
\newcommand{\cof}{ \ar@{^{(}->}[r]}

\mathchardef\dashmod="2D

\DeclareMathOperator{\colim}{colim}

\begin{document}

\title{Conjugation Spaces are Cohomologically Pure}

\author{Wolfgang Pitsch}
\address{Universitat Aut\`onoma de Barcelona \\ Departament de Matem\`atiques\\
E-08193 Bellaterra, Spain}
\email{pitsch@mat.uab.es}

\author{Nicolas Ricka}
\address{IRMA \\ Universit\'e de Strasbourg\\ 7 Rue Ren\'e Descartes, 67000 Strasbourg, France}
\email{n.ricka@unistra.fr}

\author{J\'er\^ome Scherer}
\address{EPFL \\ Institute of Mathematics\\ Station 8, CH-1015 Lausanne, Switzerland}
\email{jerome.scherer@epfl.ch}

\thanks{The authors are partially supported  by FEDER/MEC grant MTM2013-42293-P and MTM2016-80439-P. The second
author would like to thank the MPI in Bonn for its hospitality.}

\newcommand{\W}{\mathcal{W}}
\newcommand{\A}{\mathcal{A}}

\begin{abstract}
Conjugation spaces are equipped with an involution such that the fixed points have the same
mod $2$ cohomology (as a graded vector space, a ring, and even an unstable algebra) but with
all degrees divided by $2$, generalizing the classical examples of complex projective spaces
under complex conjugation. 
Using tools from stable equivariant homotopy theory we provide a characterization of conjugation spaces
in terms of purity. This conceptual viewpoint, compared to the more computational original definition, allows us
to recover all known structural properties of conjugation spaces.
\end{abstract}
\maketitle

\section{Introduction}
Given a pointed space $X$ with an involution, i.e. an action of $\C_2$, the cyclic group of order $2$, a classical way to understand $X$ is by relating 
the cohomology of $X$, of its fixed point space $X^{\C_2}$, of the orbit space $X/\C_2$ and of the space of homotopy orbits, or (reduced) Borel construction 
$X_{h\C_2}= (E\C_2{}_+ \wedge X)/\C_2$, where $\C_2$ acts diagonally on  $E\C_2{}_+ \wedge X$ (see for instance \cite{MR1236839}). 
Here we assume that $X$ has a  chosen base point, fixed by the involution $\tau$.
The space $E\C_2$ is a contractible space with a free $\C_2$-action, and its quotient is the classifying space $B\C_2$, also known as the infinite 
real projective space $\mathbb R P^\infty$.   

A \emph{conjugation space}, as introduced by Hausmann, Holm, and Puppe in \cite{MR2171799}, is an instance where this relationship is particularly well-behaved. 
Let us denote by $\HH$ the Eilenberg-MacLane spectrum representing ordinary cohomology with coefficients in the field $\F$ of two elements. 
To emphasize the role of the spectrum in the definition of cohomology, as this will be central later on, we will denote by $\HH^\ast(X)$ 
the ordinary reduced mod $2$ cohomology group of $X$. 
The Borel cohomology $\HH^\ast(X_{h\C_2})$ comes with  two natural \emph{restriction morphisms}:
\begin{itemize}
\item restriction to ordinary cohomology $\rho\colon  \HH^\ast(X_{h\C_2}) \rightarrow \HH^\ast(X)$ induced by the natural inclusion
$\C_2{}_+ \wedge X \hookrightarrow E\C_2{}_+ \wedge X$. 

\item restriction to the Borel cohomology of the fixed points $$r\colon\HH^\ast(X_{h\C_2}) \rightarrow \HH^\ast((X^{\C_2})_{h\C_2}).$$
\end{itemize}

Since $\C_2$ acts trivially on the fixed points $X^{\C_2}$, the Borel construction $(X^{\C_2})_{h\C_2}$ is the smash product
$B\C_{2+} \wedge X^{\C_2}$, and the classical K\"unneth theorem tells us that the Borel cohomology $\HH^\ast((X^{\C_2})_{h\C_2})$, 
as a graded ring, is isomorphic to $\HH^\ast(X^{\C_2})[b]$, a polynomial ring  in one variable $b$  of cohomological degree $1$ with coefficients in 
the ordinary cohomology of~$X^{\C_2}$. The original definition by Hausmann, Holm, and Puppe reads now as follows when adapted to our
pointed setup.

\begin{definition}
\label{def:conjspaceHHP} \cite[Section~3.1]{MR2171799}
{\rm A \emph{conjugation space} is a $\C_2$-space equipped with an $\HH^*$-\emph{frame} $(\kappa_0,\sigma)$, i.e.
\begin{enumerate}
 \item[a)] an additive isomorphism $\kappa_0 \colon \HH^{2\ast}(X) \rightarrow \HH^\ast(X^{\C_2})$
dividing degrees by~$2$,

\item[b)] an additive section $\sigma \colon \HH^{2\ast}(X) \rightarrow \HH^{2\ast}(X_{h\C_2})$
of the restriction map $$\rho \colon \HH^{2\ast}(X_{h\C_2}) \rightarrow \HH^{2\ast}(X),$$
\end{enumerate}

which satisfy the \emph{conjugation equation}:
\[
r \circ \sigma(x) = \kappa_0(x) b^m + lt_m
\]
for all $x \in \HH^{2m}(X)$ and all $m \in \mathbf{N}$, where $lt_m$ is a polynomial in the variable $b$ of 
degree strictly less than $m$.}
\end{definition}

The $\HH^\ast$-frame has many nice properties, as explained in the first sections of~\cite{MR2171799}:
\begin{enumerate}
	\item The morphisms $\kappa_0$ and $\sigma$ in an $\HH^*$-frame are ring homomorphisms.
	\item The $\HH^\ast$-frame is functorial for maps between conjugation spaces; in particular, if it exists, a frame is unique.
\end{enumerate}

Even more is true. Let us expand the conjugation equation by explicitly labeling its coefficients for $x$ a cohomological class of degree $2m$:
\[
r \circ \sigma(x) =  \sum_{i=0}^m\kappa_i(x) b^{m-i}.
\]
Franz and Puppe  studied in \cite{MR2198191} the behavior of the frame under the action the Steenrod algebra.
They obtained two formulas,  first for any $x \in \HH^{2m}(X; \bF)$ and any $\ell \geq 0$, one has $\kappa_0 (Sq^{2\ell} x) = Sq^\ell \kappa_0(x)$.
In the second they expressed the higher classes $\kappa_i(x)$ in terms of $\kappa_0(x)$ and Steenrod operations, namely
$\kappa_\ell(x) = Sq^\ell\kappa_0(x)$, see \cite[Theorem~1.3]{MR2198191}.

\medskip

This compatibility has many interesting properties, for instance it implies that for a conjugation \emph{manifold} $M$ the non-equivariant  cobordism 
class of $M$ is determined by that of its real locus $M^{\C_2}$. More precisely the Stiefel-Whitney classes of $M$ and of $M^{\C_2}$ determine each other
as investigated in \cite[Theorem~A.1]{MR3082744}. 

In this article we address the question whether a conjugation frame is purely algebraic or if the maps $\kappa_0$ and $\sigma$ have some
geometric meaning. Even if one can construct ``exotic'' conjugation spaces, which we do in a separate paper \cite{zoo}, the best known and most
common examples of conjugation spaces are cellular, in the sense that they arise from conjugation spheres,  \cite[Example~3.6]{MR2171799},
by attaching conjugation cells. The two-dimensional sphere $S^{1+\alpha}$ is the one-point compactification of the field of complex numbers $\CC$ 
endowed with complex conjugation and higher, even dimensional, conjugation spheres are obtained analogously from $\CC^n$. Conjugation cells
are simply unit balls in $\CC^n$ and attaching maps are required to be equivariant. Such cellular conjugation spaces are called \emph{spherical}
in \cite[Section~5.2]{MR2171799}. For instance the classifying space $BU$ with the complex conjugation is a spherical conjugation complex, and 
this allowed us to rather straightforwardly develop a theory of equivariant Stiefel-Whitney classes for Real vector bundles,~\cite{MR3082744}.

To understand how close arbitrary conjugation spaces are from being spherical, we follow the guiding principle brought by the second author
and recast the definition of conjugation spaces  in the equivariant stable world. The main advantage of this approach is that the various restriction maps,
the halving isomorphism $\kappa_0$, and the section $\sigma$, are directly encoded in the graded Mackey functor structure of equivariant mod $2$
cohomology. Our main references are Hill, Hopkins, and Ravenel's \cite{MR3505179}, Greenlees and May's \cite{MR1230773}, and \cite{MR1413302}.
We work with the equivariant Eilenberg-Mac Lane spectrum $\HF$, associated to the Mackey functor $\underline{\mathbf{F}}$ 
and whose associated cohomology theory is called ordinary equivariant cohomology. Recall that for any space $X$ the smash product
$X \wedge H$ splits as a wedge of Eilenberg-Mac Lane spectra, \cite[Lemma~II.6.1]{MR1324104}. Moving to the equivariant world, this is still 
the case for finite $\C_2$-spaces as shown by C.~May in \cite{Clover}. We follow \cite[Definition~4.56]{MR3505179} and introduce the notion of purity.

\medskip
\noindent {\bf Definition~\ref{def:homolpurity}.}
An equivariant space is \emph{homologically pure} if there exist a set $I$, natural numbers $n_i$ for any $i \in I$, and a 
weak equivalence of right $\HF$-modules
\[
 X\wedge \HF \simeq \bigvee_{i \in I}\Sigma^{n_i(1+\alpha)}\HF.
\]

\medskip

The main result of the present work is a stable equivariant characterization of conjugation spaces in terms of purity. We impose a mild finiteness
condition: We say that a space is of \emph{finite type} if its ordinary mod $2$ cohomology is finitely generated in each degree.

\medskip
\noindent {\bf Theorem~\ref{thm:rickadef}.}
{\it
Let $X$ be a $\C_2$-space of finite type. Then $X$ is a conjugation space if and only if it is homologically pure. 
}

\medskip

In this definition there is no mention of the section $\sigma$, the degree halving isomorphism $\kappa_0$ or the conjugation equation. 
We will show that both maps are in fact  induced by precise  geometric maps in equivariant cohomology, which explains geometrically 
the unicity in the conjugation frame. The compatibility of these maps with cohomological operations in $\HF$-cohomology that preserve  
the  line $\{m(1 + \alpha \ | \ m \in \mathbb{Z}) \}$ implies their compatibility  with Steenrod squares. This provides a conceptual proof
of the Franz-Puppe result mentioned above. Following an indication by Lannes we also show that for a conjugation space $X$, 
the Borel cohomology $\HH^\ast(X_{h\C_2})$ is functorially determined by  $\HH^\ast(X^{\C_2})$. Let us conclude this introduction by 
mentioning Olbermann's alternative 
definition of conjugation spaces in \cite[Remark~2.4]{MR2425140}. He refers to this as a definition without a conjugation equation, as it does not refer explicitly to the existence of an $\HH^\ast$-frame. Whereas his viewpoint is algebraic ours is more geometric.

\medskip

Here is a short outline of this paper. We recall in Section~\ref{sec:equivariant} some features of 
equivariant spectra and equivariant cohomology theories. Here and in the next two short sections we fix various notations,
and present results about the geometric fixed points, as well as  Stong's computation 
of the coefficients of $\C_2$-equivariant ordinary cohomology.
We present the stable equivariant background in some details because the literature on  conjugation spaces has been mostly written
unstably up to now. These first sections can be safely skipped by equivariant experts, but they will be useful to other readers and save 
them the need to go through many references. Then in Section~\ref{sec:pureimpliesconj} we show the first half of our result, namely that homologically pure 
spaces  of finite type are 
conjugation spaces. In Section~\ref{sec:conjimpliespure} we prove the reverse implication for finite type conjugation spaces.
Finally in Section~\ref{sec:properties} we use our definition to exhibit some properties of conjugation spaces, we prove in particular the compatibility of the conjugation 
frame with Steenrod operations and show that $\HH^\ast(X_{h\C_{2}})$ and $\HH^\ast(X)$ are related via the ``derived functor of the destabilization functor'' 
of Lannes-Zarati~\cite{MR871217}. The results of this last section depend on explicit computations involving the equivariant Steenrod algebra which might be of independent interest and are the subject of Appendix~\ref{app:compcohomology}.

\subsection*{Acknowledgments} During the years that lead to the present work, we where encouraged by a number of people. 
We specially thank J.-C. Hausmann and I.~Hambleton for their continued interest and support, J.~Lannes for kindly pointing out to us the relationship 
between conjugations spaces and his work on the derived functor of the destabilization, and I.~Patchkoria for enlightening discussions about
geometric fixed points and Steenrod operations.  

We warmly thank the referee for pointing out several important issues and indicating useful references, his remarks lead us to greatly improve the exposition.

\section{Equivariant spectra and cohomology}
\label{sec:equivariant}
In all this work we will denote by $\C_2 = \{ e, \tau \}$ the cyclic group of order two, where $e$ stands for the neutral element,  and by $\bF$ the field with two elements. 
By convention a $\C_2$-space $X$ is a topological space  with a specific choice of an involution given by the 
action of the generator~$\tau$. By analogy with the conjugation action on the complex numbers, the subspace of fixed points $X^{\C_2}$ will be called 
the \emph{real locus} of~$X$.

We have tried to follow a coherent notation in this article. We have been helped by Greenlees's \cite{MR3807754}, even if the encounter of
stable equivariant homotopy theory with conjugation spaces sometimes lead us to make different choices.

\subsection{Equivariant spectra}\label{subsec:equivspec}

 We take a stable approach to (ordinary) cohomology
since  Brown Representability  says precisely that a generalized cohomology theory is represented by a spectrum, see
\cite[Theorem~9.27]{MR1886843}. Identifying homotopy equivalent spectra is then a natural step as they represent the same cohomology theory.

Let us denote by $\T$ the category of \emph{pointed} topological spaces, by $\Sp$ the \emph{pointed} category of spectra.
By construction the categories of topological spaces and of spectra are related by a pair of adjoint functors: 
\[
\Sigma^\infty \colon \T \leftrightarrows \Sp\colon \Omega^\infty.
\]

Let us turn to the equivariant case. Everything will be stated for the group $\C_2$ but almost all the aspects we discuss 
here are true in a much larger generality, see for instance the nice introduction by Greenlees and May to the subject in 
the Handbook \cite{MR1361893}, the monograph on Tate cohomology by the same authors, \cite{MR1230773},  or the classical \cite{MR866482}.

Denote by $\C_2\T$ the category of pointed topological spaces endowed with a $\C_2$-action, where
the morphisms are the equivariant maps, and the weak equivalences are the equivariant weak equivalences
as defined in Definition~\ref{def:equivequivalence}.
A new feature of the equivariant category $\C_2\T$ is that there is more than one equivariant sphere in each dimension 
with respect to which one may suspend. 

 \begin{definition}\label{def:representationsphere}
Given any finite dimensional orthogonal representation $V$, the \emph{representation sphere} $S^V$ is the one point 
compactification of $V$. If $\mathbb{R}^n$ is a trivial representation then we simply write $S^n$ for $S^{\mathbb{R}^n}$.

By $S(V)$ we denote the \emph{unit sphere} in $V$, endowed with the restriction of the action of $\C_2$ on $V$.
\end{definition}

The space $S^V$ is a sphere of dimension $\dim V$, with a canonical base point, namely the image of $0$, which is fixed under
the action of $\C_2$. For each such sphere and  any $\C_2$-space $X$ we may consider the smash product $S^V \wedge X$ 
with diagonal action. Like in the non-equivariant 
case the passage from $\C_2\T$ to equivariant spectra $\C_2\Sp$ amounts to inverting all operations 
$S^V \wedge - : \C_2\T \rightarrow  \C_2\T$.
The categories of equivariant pointed topological spaces and spectra are again related by a pair of adjoint functors:
\[
\Sigma^\infty_{\C_2} \colon \T \rightleftarrows \Sp \colon \Omega^\infty_{\C_2}.
\]

In the development of stable homotopy theory much effort has been put in providing structured models for the homotopy category of spectra. 
The symmetric monoidal closed structure on spaces (either equivariant or non-equivariant) induces a symmetric monoidal closed structure on 
the chosen categorical model for equivariant  spectra (see \cite{MR1922205,MR1806878}). Whenever explicitly needed we will use in this work the  model of orthogonal spectra, see~\cite[Definition~III.1.7]{Schwede}.

\begin{definition} \label{def:symmonclosed}
We denote by $\wedge$ and $F_{\C_2}(-,-)$ the \emph{monoidal product} and the \emph{equivariant function 
spectrum} respectively. For non-equivariant spectra, we denote simply by $F(-,-)$ the function spectrum.
\end{definition}

Since $\wedge$ endows the stable category with a symmetric closed monoidal product, 
it makes sense to talk about ring objects, i.e. spectra $R$ with a multiplication $R \wedge R \rightarrow R$ making the usual diagrams 
commute, and (right) $R$-module spectra $M$ endowed with an action $M \wedge R \rightarrow M$. Observe in particular that for any spectrum $M$ the multiplication in $R$ gives $M \wedge R$ a canonical $R$-module structure. 

\begin{definition}
\label{def:module}
Let $E$ be a ring spectrum. We denote by $E\dashmod \mod$ the category of (right) \emph{$E$-modules}. If $E$ is commutative, 
the smash product of spectra induces a symmetric monoidal closed structure on $E\dashmod \mod$. 
We denote by $\wedge_{E}$ and $F_{E\dashmod \mod}(-,-)$ the corresponding tensor product and internal hom respectively.
\end{definition}

There are two functors that help us to relate the equivariant stable homotopy category with the standard one.
We have first the \emph{restriction} functor we get by forgetting the action:
\[
\begin{array}{rcl}
\C_2\Sp & \longrightarrow & \Sp \\
X & \longmapsto & X^u 
\end{array}
\]
and the \emph{trivial action} functor which allows us to include ordinary spectra into equivariant ones:
\[
\begin{array}{rcl}
\Sp & \longrightarrow & \C_2\Sp \\
X & \longmapsto & \iota X 
\end{array}
\]

Both functors induce triangulated functors on the homotopy category and preserve the smash product, i.e. they are strongly monoidal.
They also preserve compact objects and products, \cite[Section~II.4]{MR866482}.  
As a consequence, see \cite{MR866482}, or the derived and very general viewpoint \cite[Theorem~1.7]{MR3542492}, 
both functors are part of a series of adjunctions. Most notably, we have a first series of adjunctions:
\begin{equation}
\label{adj:free}
\C_{2+} \wedge - \dashv (-)^u \dashv F(\C_{2+},-).
\end{equation}
The leftmost adjoint is the \emph{free action} functor:
\[
\begin{array}{rcl}
\Sp & \longrightarrow & \C_2\Sp \\
X & \longmapsto & \C_{2+} \wedge  X 
\end{array}
\]
where the action is induced by the left action on $\C_{2+}$,
and the right adjoint:
\[
\begin{array}{rcl}
\Sp & \longrightarrow & \C_2\Sp \\
X & \longmapsto &  F(\C_{2+},X) 
\end{array}
\]
is given by the function spectrum on which $\C_2$ acts on the left through its right action on  itself, \cite[Section~II.4]{MR866482}.

The functor $\iota$ also admits both left and right adjoints. That the right adjoint
	\begin{equation}
	\label{adj:fix}
	 \iota \dashv (-)^{\C_2}.
	\end{equation}
	 can be identified with a fixed-point functor is a result of Lewis~\cite{Lewis95}. The left adjoint can not be identified with taking the orbits in general, as is the 
	 case for spaces. It is the case in some very specific situations, see \cite[Chapter~II]{LMS86}. We will need this in one such circumstance, and come to this at the appropriate time.

One of the subtleties in the theory is the interaction of these functors with the monoidal structure.
Most notably the fixed points functor and the equivariant suspension $\Sigma^\infty_{\C_2}$ \emph{ do not} commute, 
even for the sphere by tom Dieck's splitting Theorem (see the original reference \cite[Satz 2]{tD75}, or \cite[Section V]{LMS86}). 
The introduction of the geometric fixed points, see Definition~\ref{def:geofixedpoints} below, is useful to tackle this issue. 

\subsection{Mackey-valued cohomology }\label{subsec:mackeyfunc}

Given an ordinary spectrum $E$ we denote the associated cohomology and homology  theories evaluated at a spectrum $X$ by
\[
E^\ast(X) = [S^{-\ast} \wedge X,E]
\ \ \textrm{and} \ \
E_\ast(X) = [S^{\ast} , X\wedge E]
\]
where $\ast \in \mathbb{Z}$, and $[-,-]$ denotes stable homotopy classes of maps. 
If $X$ is  a space  we will freely confuse it with its  suspension spectrum $\Sigma^\infty X$ if
this is clear from the context.

There is a conceptual explanation stemming from the structure of the homotopy  category $h\Sp$ as to why ordinary cohomology takes value 
in abelian groups and this is related to the $t$-structure arising from the notion of connectivity.
Denote by $h\Sp_{\geq 0}$ the subclass of connective spectra, i.e. such that $\pi_n(X) = 0$ for $n< 0$, and by $h\Sp_{\leq -1}$ 
the co-connective spectra. Then the \emph{heart} of this structure $h\Sp_{\geq 0} \cap \Sigma h\Sp_{\leq -1}$ is isomorphic to the 
category of abelian groups. The spectrum corresponding to the abelian group $A$ is the Eilenberg-MacLane spectrum $\HH A$, 
characterized by the fact that:
\[
\pi_n(\HH A) = \left\{\begin{array}{lc} A & \textrm{ if } n = 0, \\ 0 & \textrm{ if } n \neq 0. \end{array} \right.
\]
The same construction holds true in essence when taking into account a $\C_2$-action. 

\begin{definition}\label{def:equivequivalence} \cite[Definition~I.4.4]{MR866482}
A map $f\colon X \rightarrow Y \in \ZS$ is a \emph{weak equivalence} in $\ZS$ if for any $n \in \Z$ the morphisms
\begin{enumerate}
\item $\pi_{n}^{\C_2}(f)\colon \pi_n^{\C_2}(X) = [S^n,X]^{\C_2} \rightarrow [S^n,Y]^{\C_2} = \pi_n^{\C_2}(Y)$, and

\medskip

\item $\pi_n^e(f)\colon \pi_n^e(X) = [\C_{2+}\wedge S^n,X]^{\C_2}\rightarrow  [\C_{2+}\wedge S^n,Y]^{\C_2} = \pi_n^e(Y)$,
\end{enumerate}
are both isomorphisms, where $[-, -]^{\C_2}$ indicates homotopy classes of equivariant maps.
\end{definition}

For each $n$ the functors $\pi_{n}^{\C_2}$ and $\pi_{n}^{e}$ are part of a richer structure, namely a Mackey functor $\underline\pi_n$, which
we will introduce below in Subsection~\ref{subsec:Mackeyfunctors}. Notice that $\pi_n^{\C_2}(X) \cong \pi_n X^{\C_2}$ since the sphere
$S^n$ is endowed with the trivial action, see \eqref{adj:fix}, and $\pi_n^{e}(X) \cong \pi_n X^{u}$ by the free-forgetful adjunction \eqref{adj:free}.
As in the non-equivariant case, define an equivariant spectrum to be \emph{$k$-connected} if, for any $n\geq k$, we have $\underline{\pi}_n(X) = 0$, 
and $k$-coconnected if for any $n\leq k$, we have $\underline{\pi}_n(X) = 0$. Then, exactly as for ordinary spectra, one can identify the heart of the associated
$t$-structure, see \cite[Proposition~I.7.14]{MR866482}.

\begin{proposition}
\label{prop:heartisMackey}
The heart of the $t$-structure determined by the classes of connective $\ZhS_{\geq0}$ and coconnective  $\ZhS_{\leq -1}$ spectra is 
isomorphic to the abelian category of Mackey functors $\mathcal{M}$. 
\end{proposition}

In particular we have an Eilenberg-Mac Lane functor $\HH\colon \mathcal{M} \rightarrow \ZhS$, taking an abelian group valued Mackey functor 
to an equivariant spectrum. Let us now explain briefly what Mackey functors are.

\subsection{Mackey functors}\label{subsec:Mackeyfunctors}

The theory of Mackey functors is very rich and we refer the interested  reader for instance to~\cite{MR1261590}; 
here we will only recall the specifics of Mackey functors for the cyclic group $\C_2$; this simpler description is extracted from the ``elementary approach" given by Ferland and Lewis~\cite[Section 1.1]{MR2025457}. Consider the additive category  $\mathcal{O}$, whose objects consists in the two transitive $\C_2$-sets, namely $\C_2/\C_2 = pt$, and $\C_2/ e = \C_2$.

The abelian groups of morphisms are generated by the identities, a \emph{transfer} map $\tr\colon \C_2 \rightarrow pt$, 
a \emph{restriction} morphism $\rho\colon pt \rightarrow \C_2$, and $\theta\colon \C_2 \rightarrow \theta$, pictured as follows, in what one can call a Lewis diagram:
\[
	\xymatrix{
		pt  \ar@/^2.0ex/[d]^{\rho}\\ 
		\C_2 \ar@/^2.0ex/[u]^{\tr} \ar@(dl,dr)^{\theta}
		}
	\]
The map $\theta$ is given by the transitive action. They are subject to the following relations:
 
\begin{enumerate} 
    \item 	$\tr \theta  = \tr$, 
    \item $\theta \rho = \rho$, 
    \item $\theta^2= Id$, 
    \item $\rho \tr = 1 + \theta$.
		\end{enumerate}

\begin{definition}\label{def:mackeyfucnt}
A \emph{Mackey functor} (for the cyclic group $\C_2$) is an additive functor $M\colon \mathcal{O} \rightarrow\mathcal{A}b$
with values in the category of abelian groups.
\end{definition}

Note that because of $(3)$, the group $M(\C_{2})$ is a $\C_{2}$-module, and $(1)$ and $(2)$ express the fact that $\rho$ and $\tr$ 
are $\C_2$-equivariant, where $M(pt)$ has a trivial action. Finally, from $(4)$ we recover from the restriction and transfer the  action of $\C_2$ on $M(\C_2)$ 
since $\theta$ is given by the map $ \rho \tr -1$. For this reason we omit the action of the map $\theta$ in the description of the Mackey functors 
in Section~\ref{sec:ordinary} hereafter.

\begin{example}\label{ex:equivariantpi}
The equivariant homotopy groups we have introduced above form a Mackey functor $\underline\pi_n$. Explicitly,
$(\underline\pi_n)_{pt} (X) = \pi_n^{\C_2}(X) =[S^n,X]^{\C_2}$ and $(\underline\pi_n)_{\C_2} (X) = \pi_n^{e}(X)=[S^n,X^u]$. The restriction
morphism is given by forgetting the $\C_2$-action. This integral grading can be extended to an $RO(\C_2)$-grading by setting 
$\pi_V^{\C_2}(X) = [S^V,X]^{\C_2}$ and $\pi_V^e(X)=[S^{\dim V},X^u]$ for any 
$V \in RO(\C_2)$.  
\end{example}

\begin{convention}
	\label{conv:rep}
	\begin{enumerate}
		\item We will denote by $1$ the trivial representation and by $\alpha$ the sign representation of $\C_2$. 
		These two representations freely generate  the representation ring $RO(\C_2)$ as an abelian group.
		
		\item Given an element $n + m \alpha \in RO(\C_2)$, its \emph{dimension} is the integer $n+m$.
		
		\item In general an $RO(\C_2)$-grading of an object will be emphasized by $\star$ and an integral grading by $\ast$. 
		For any homogeneous element $x$, its degree will be denoted by $|x|$, similarly for $n+m\alpha \in RO(\C_2)$, its dimension 
		will be denoted by $|n+m\alpha|$.  
	\end{enumerate}
\end{convention}

\begin{definition}\label{def:equivcoho}
Given a spectrum $E\in \ZS$, the associated $E$-\emph{cohomology theory} it represents is denoted $\underline{E}$ and given as follows: 
for $X \in \ZS$, $\star \in RO(\C_2)$
\begin{enumerate}
\item $ \underline{E}^\star_{pt}(X) = [ S^{-\star} \wedge X,E]^{\C_2}$, and

\medskip

\item $\underline{E}^\star_{\C_2}(X) = [S^{-\star} \wedge X \wedge \C_{2+}, E]^{\C_2}= [S^{-|\star|} \wedge X , E^u]$.
\end{enumerate}
\end{definition}

From now on we will  denote simply $E^\star(-) = \underline{E}^\star_{pt}(-)$.

\section{Geometric fixed points}
\label{sec:geometric}
A basic tool in equivariant homotopy theory is the ``isotropy separation sequence", \cite[(2.44)]{MR3505179}, which we review first.
The rest of the section is then devoted to geometric fixed points and $a$-periodicity.

\subsection{Isotropy separation sequence}\label{subsec:isotsepseq}
For each $n \in \mathbb{N}$ we have a morphism of cofiber sequences in $\ZT$, see Definition~\ref{def:representationsphere}
for the notation we use for spheres:
\begin{equation}
\label{eq:cof}
\xymatrix{
S(n\alpha)_+ \ar[d] \ar[r] &  S^0 \ar[r] \ar@{=}[d]&  S^{n\alpha} \ar[d]\\
S((n+1)\alpha)_+ \ar[r] & S^0 \ar[r] & S^{(n+1)\alpha}
}
\end{equation}
The action on $S^0$ is trivial and vertical colimits fit into a cofiber sequence:
\[
E{\C_2}_+ \rightarrow S^0 \longrightarrow \widetilde{E\C_2}
\]

\begin{remark}\label{rem:ring}
The spaces $\widetilde{E\C_2}$ and $E\C_2$ are characterized by the following universal properties:
\[
(E\C_2)^u \simeq pt, \  \quad  E\C_2^{\C_2} = \emptyset \  \quad \textrm{and} \  \quad \widetilde{E\C_2}^u \simeq pt, \  \quad \widetilde{E\C_2}^{\C_2} = S^0.
\]
In particular this characterization implies that $\widetilde{E\C_2} \wedge \widetilde{E\C_2} \simeq \widetilde{E\C_2}$ and we can use the equivalence  to define a multiplication turning 
$\widetilde{E\C_2}$ into a ring spectrum.
\end{remark}

By smashing the above cofiber sequence with a space $X \in \ZT$, we obtain a new cofibration sequence.

\begin{lemma}\label{em:univspaces}
For any $\C_2$-space $X$ there is a cofibration sequence
\[
E{\C_2}_+ \wedge X \rightarrow  X \longrightarrow \widetilde{E\C_2} \wedge X
\]
and for any map $f\colon X \rightarrow Y$ in $\ZT$, the map obtained by forgetting the action $f^u\colon X^u \rightarrow Y^u$ 
is a weak equivalence in $\T$ if and only if $E{\C_2}_+ \wedge f$ is a weak equivalence in $\ZT$. \hfill{\qed}
\end{lemma}

Pushing the cofiber sequence given by the previous lemma in $\ZS$ we get the \emph{isotropy separation sequence}. 
For any $X \in \ZS$ there is a cofiber sequence:
\[
E{\C_2}_+ \wedge X \rightarrow  X \longrightarrow \widetilde{E\C_2} \wedge X
\]

The isotropy separation sequence in $\ZS$ has the key property that it separates a space into a  free part and a singular part, this is the
slogan in \cite[page~2]{MR1361893}. 
Indeed, since after forgetting the $\C_2$-action $E\C_2$ is contractible, and the forgetful functor sends cofiber sequences to cofiber sequences, 
for any $X \in \ZS$, $(\widetilde{E\C_2} \wedge X)^u$ is contractible, and given any morphism $f\colon \widetilde{E\C_2} \wedge X \rightarrow Y$ in $\ZS$, 
$f$ is a weak equivalence if and only if $f^{\C_2}\colon (\widetilde{E\C_2} \wedge X)^{\C_2} \rightarrow Y^{\C_2}$ is a weak equivalence in $\Sp$.

 \subsection{Geometric Fixed points}
 \label{subsec:Phi} 
 The fixed-points functor is not monoi\-dal on spectra, but there is a better behaved related functor, namely 
 the so-called  geometric fixed point functor, see \cite[Subsection~2.5.2]{MR3505179}.
 
 \begin{definition}\label{def:geofixedpoints}
The \emph{geometric fixed points} functor $\Phi^{\C_2}\colon \ZS \longrightarrow \Sp$ is given by 
 $\Phi^{\C_2}(X) = (\widetilde{E\C_2} \wedge X)^{\C_2}$.
 \end{definition} 
 
 The associated derived functor $\Phi^{\C_2} \colon h\ZS \longrightarrow h\Sp$ is abusively denoted by the same symbol and this is the one
 we will systematically use in this article. It enjoys many nice properties.
 
 \begin{proposition}\label{prop:propgeofixedpoints}
 	The functor $\Phi^{\C_2}$ has the following properties:
 	\begin{enumerate}
 		\item $\Phi^{\C_2}$ sends weak equivalences to weak equivalences and commutes with filtered colimits.
 		\item A map $f\colon X \rightarrow Y$  in $\ZS$ is a weak equivalence if and only if $\Phi^{\C_2}(f)$ and $f^u$ are weak equivalences.
 		\item For any $X \in \ZT$ we have $\Phi^{\C_2}(\Sigma^\infty_{\C_2} X) \simeq \Sigma^\infty(X^{\C_2})$.
 		\item For any $X,Y \in \Sp$ we have $\Phi^{\C_2}(X \wedge Y) \simeq \Phi^{\C_2}(X) \wedge \Phi^{\C_2}( Y)$.
 	\end{enumerate}
 \end{proposition}
 \begin{proof}
 	Properties $(1),(3),(4)$ are directly taken from  \cite[(Proposition~2.45)]{MR3505179}. The characterization of weak equivalences is
 	a consequence of  \cite[(Remark~2.46)]{MR3505179}.
 \end{proof}

\subsection{Periodic modules}\label{subsec:permodules} 
The inclusion of fixed points gives a map $S^0 \rightarrow S^\alpha$ into the sign representation sphere, 
hence an element  $a \in \underline{\pi}_0(S^\alpha)$. The name of this map is $a_\alpha$ in \cite[Definition~3.11]{MR3505179}
or $a_\sigma$ in \cite{HillPreprint}.

This element is killed by the map $(1 +\theta)$, which corresponds to the element $[\C_2]$ in $A(\C_2) \cong \pi_0^{\C_2}(S^0)$ under the
isomorphism obtained by Segal, \cite{MR0423340}, a particular case of the tom Dieck splitting, \cite[Satz~2]{MR436177}. 
Indeed we know that this action is given by $\tr \rho -1$ (see \ref{subsec:Mackeyfunctors}). 
The equality $(1 + \theta)a=0$ now follows from the commutativity of the diagram
\[
\xymatrix{
S^0 \ar[d]^{t} \ar[dr]^{1 + \theta} & & \\
\C_{2+} \ar[r]^{\rho} & S^0 \ar[r]^a & S^{\alpha}
}
\]
where the bottom row is a cofiber sequence. 


Smashing the map $a\colon S^0 \rightarrow S^{\alpha}$ with any equivariant spectrum gives us a map $E \rightarrow S^{\alpha} \wedge E$.

\begin{definition}
\label{def:periodic}
A spectrum $E$ is $a$-\emph{periodic} if and only if the map $a \wedge E\colon E \rightarrow S^{\alpha} \wedge E$ is a weak equivalence.
\end{definition}

\begin{remark}
The class of $a$-periodic spectra appears in the literature in different forms. For instance, it corresponds to spectra which are local 
with respect to the trivial subgroup in the sense of \cite[Section 6]{MNN}. 
\end{remark}

The cohomology of an $a$-periodic spectrum is then also $a$-periodic, in the sense that the action of $a$ on the cohomology,
decreasing the $RO(\C_2)$-bidegree by $\alpha$, is an isomorphism. 
By construction, the prototypical example of an $a$-periodic spectrum is $\widetilde{E\C_2}$; 
and indeed it is the source  of all $a$-periodic spectra as the next two propositions show:

\begin{proposition}\label{prop:tildeEz2isaperiodic}
The spectrum $\widetilde{E\C_2}$ is $a$-periodic, and as a consequence  for any $X \in \ZS$, both $\widetilde{E\C_2} \wedge X$
and the functional spectrum $F_{\C_2}(\widetilde{E\C_2},X)$ are also $a$-periodic.
\end{proposition}
\begin{proof}
The first statement follows from the description of $\widetilde{E\C_2}$ as a colimit and the second from the first.
To prove the third observe that  $S^{\alpha}$ is a finite $\C_2$-CW-spectrum, hence strongly dualizable by
\cite[Theorem~XVI.7.4]{MR1413302}, and its dual is $S^{-\alpha}$. Therefore the duality map is an equivalence and
we have
\[
S^{\alpha} \wedge F_{\C_2}((\widetilde{E\C_2},X) \simeq F_{\C_2}(DS^{\alpha}, F_{\C_2}(\widetilde{E\C_2},X)) \simeq F_{\C_2}(S^{-\alpha} \wedge \widetilde{E\C_2},X)
\]
where we used \cite[Corollary~XVI.7.5]{MR1413302} for the first equivalence and the second is by adjunction.
\end{proof}

\begin{proposition}\label{prop:caracaper}
Let $E \in \C_2\Sp$. Then the following are equivalent:
\begin{enumerate}[(i)]
\item $E$ is $a$-periodic.
\item The canonical map coming from the isotropy separation sequence
$E \rightarrow \widetilde{E\C_2}\wedge E$ is an equivalence.
\item The underlying non-equivariant spectrum $E^u$ is contractible.  
\item Multiplication by $a$ is an isomorphism on the $RO(\C_2)$-graded abelian group $\pi^{\C_2}_\star(E)$.
\end{enumerate}
\end{proposition}
\begin{proof}
The equivalence $(i)\Leftrightarrow (ii)$ is an immediate consequence of the description of $\widetilde{E\C_2}$ as a homotopy colimit:
\[
S^0 \stackrel{a}{\longrightarrow} S^\alpha  \stackrel{a}{\longrightarrow} S^{2\alpha} \rightarrow \cdots \rightarrow S^{n\alpha} \stackrel{a}{\longrightarrow} 
S^{(n+1)\alpha} \rightarrow \cdots
\] 
To prove $(ii)\Leftrightarrow (iii)$, observe that if $E \rightarrow \widetilde{E\C_2}\wedge E$ is an equivalence, 
then, as $(\widetilde{E\C_2} \wedge E)^u$ is contractible, see Remark~\ref{rem:ring}, 
so is $E^u$. Conversely we want to show that $E \rightarrow \widetilde{E\C_2} \wedge E$ is an equivalence. 
By assumption $E^u \rightarrow (\widetilde{E\C_2} \wedge E)^u$ is an equivalence, and since $S^0 \wedge  \widetilde{E\C_2}  \stackrel{\sim}{\rightarrow} 
\widetilde{E\C_2}  \wedge  \widetilde{E\C_2}$, we certainly have $\Phi^{\C_2}(E) \simeq \Phi^{\C_2}(\widetilde{E\C_2} \wedge E)$. We conclude  
by Proposition~\ref{prop:propgeofixedpoints} (2).  

Finally, condition $(iii)$ is equivalent to $(i)$ which implies $(iv)$, since a weak equivalence induces an isomorphism 
of homotopy groups.
Conversely, if $(iv)$ holds, apply the functor $[-,E]$ to the cofiber sequence $\Sigma^n\C_{2+} \rightarrow S^n \rightarrow S^{n+\alpha}$
from (\ref{eq:cof}) where the unit sphere $S(\alpha)$ has been identified with $\C_2$. This yields a long exact sequence
\[
\cdots \stackrel{a}{\rightarrow} \pi_{n}^{\C_2}(E) \rightarrow \pi_{n}^e(E) \rightarrow \pi_{n-1+\alpha}^{\C_2}(E) \stackrel{a}{\rightarrow}\cdots.
\]
Consequently, $(iv)$ implies that $\pi_n^e(E) = 0$ for all $n$. 
\end{proof}

We have seen in Remark~\ref{rem:ring} that, by construction, $\widetilde{E\C_2}$ is a ring spectrum and we can consider the category of $\widetilde{E\C_2}$-modules. The following corollary of Proposition~\ref{prop:caracaper} 
gives an interpretation of the $a$-periodization in these terms.

\begin{corollary}
\label{cor:aperleftadj}
The full subcategory $\C_2\Sp[a^{-1}]$ of $a$-periodic $\C_2$-equivariant spectra is equivalent to that of $\widetilde{E\C_2}$-modules.
The $a$-periodization functor, left adjoint to the forgetful functor, is given by smashing with $\widetilde{E\C_2}$.
\end{corollary}
\begin{proof}
An $a$-periodic spectrum is an $\widetilde{E\C_2}$-module  up to homotopy by Proposition~\ref{prop:caracaper}~$(ii)$.
The adjunction we are studying here is thus simply the free-forgetful adjunction 
$S^0-\hbox{\rm mod} \rightleftarrows \widetilde{E\C_2}-\hbox{\rm mod}$.
\end{proof}

We have an equivalence $(\widetilde{E\C_2} \wedge X) \wedge \HF \simeq \widetilde{E\C_2} \wedge (X \wedge \HF)$
for any spectrum $X$ by associativity of the smash product. In other words:

\begin{lemma}\label{lem:aperprop}
The $a$-periodization functor commutes with the free (right) $\HF$-module functor.
\hfill{\qed}
\end{lemma}

\section{The structure of $\HF$}
\label{sec:ordinary}

In this article we work with ordinary equivariant cohomology with constant coefficients $\F$, which means
that we will use as our cohomology-defining spectrum  the Eilenberg- Mac Lane spectrum $\HH \underline \F$, where $\underline \F$ is the Mackey functor with Lewis diagram:
\[
\xymatrix{
	\F \ar@/^2.0ex/[d]^{=}\\ 
	\F \ar@/^2.0ex/[u]^{0}
}
\]
We record in this section  some of its most relevant properties for the present work for future reference and completeness. The proofs are classical.



\begin{proposition}
\label{lem:fix}
\begin{enumerate}
\item There are weak equivalences $(\HF)^{\C_2} \simeq \HH$ and $(\HF)^{u} \simeq \HH$.

\item There is a unique map $\HF \wedge \HF \rightarrow \HF$ giving $\HF$ the structure of a commutative ring spectrum.
\end{enumerate}
\end{proposition}

We prepare now for the computation of the coefficients $\HF^\star$, the $RO(\C_2)$-graded value of the Mackey functor at the trivial orbit $pt$, evaluated on $S^0$.
Together with the Euler class $a$ introduced in Section~\ref{subsec:permodules}, the orientation class will help us organize the data.

\begin{definition}\label{def:orientation}
The \emph{orientation class} is the only non-trivial element $u \in \HF^{\alpha-1} = [S^{1-\alpha},\HF]^{\C_2} \cong \F$.
\end{definition}

The orientation class is the mod $2$ analog of the class called $u_\alpha$ in \cite[Definition~3.12]{MR3505179}, where it is constructed
at the level of equivariant cellular chains.

Let us now recall some known facts about the algebraic structure of the ring $\HF^\star$ and the Mackey functor $\underline{\HF}^\star$. 
The first proposition highlights the role played by the two classes $a$ and $u$ for the structure of the ring $\HF^\star$. From now on, we identify the element $a$ with the element also denoted by $a$ in $\HF^{-\alpha}$ that corresponds
to the composite $S^0 \rightarrow \HF \simeq S^0 \wedge \HF \xrightarrow{a \wedge \HF} S^\alpha \wedge \HF$.
There is also a more geometric approach in \cite{HillPreprint}.
Recall that if $R$ is a commutative ring and $M$ an $R$-module 
the square-zero extension of $R$ by $M$ is the ring whose underlying $R$-module is $M \oplus R$ with ring structure given by $(m,r)\cdot(n,s)= (rn + sm,rs)$.

\begin{proposition} \cite[Proposition 6.2]{HK}
\label{prop:algstrHF}
The ring $\HF^\star$ has the structure of a square-zero extension of the polynomial ring $\bF[a,u]$ by its module $M=u^{-2}\bF[a^{-1},u^{-1}]$.
\end{proposition}

Proposition~\ref{prop:algstrHF} gives the structure of the ring $\underline{\HF}^\star_{pt}$. To identify the Mackey functor restriction for $\HF$
we use the free-forgetful adjunction (\ref{adj:free}).

\begin{lemma}
\label{lemma:restriction}
The restriction  $\rho\colon \underline{\HF}^\star_{pt} (X) \rightarrow \underline{\HF}^\star_{\C_2} (X)$ coincides with the
morphism induced by forgetting the action ${\HF}^\star (X) \rightarrow \HH{\bF}^{\mid\star \mid} (X^u)$. As a ring
$\underline{\HF}^\star_{\C_2} (X) \cong \HH{\bF}^{*} (X^u)[u^{\pm 1}]$.
\end{lemma}

\begin{proof}
By definition $\underline{\HF}^\star_{\C_2} (X) = [S^{-\star} \wedge X \wedge \C_{2+}, \HF]^{\C_2}$. 
Since the shearing map $S^{-\star} \wedge X \wedge \C_{2+} \rightarrow (S^{-\star} \wedge X)^u \wedge \C_{2+}$ is an equivariant weak equivalence, 
we can use the above mentioned adjunction and Proposition~\ref{lem:fix}(1) to identify
\[
\underline{\HF}^\star_{\C_2} (X) \cong [(S^{-\star} \wedge X)^u \wedge \C_{2+}, \HF]^{\C_2} \cong [(S^{-\star} \wedge X)^u, \HH]
\]
The ring structure comes for free since $u$ lives in degree $1-\alpha$ and the map induced by $u$ after  forgetting the action  is non-trivial, hence represented by
the only available unit map $S^0 \rightarrow \HH$.
\end{proof}

Except for the more modern notation $\underline{\F}$, we use the same symbols as in \cite{MR2025457} for 
the four Mackey functors that will appear in the next proposition:
\[
\xymatrix{   \textrm{Functor symbol} &\bullet & \underline{\F} &    L & L_-\\ 
 \textrm{Lewis diagram} & \bF \ar@/^/[d]	& \bF \ar@/^/[d]^{=} &  \bF \ar@/^/[d]^{ 0} & 0 \ar@/^/[d]  \\ 
 & 0 \ar@/^/[u] & \bF \ar@/^/[u]  &  \bF \ar@/^/[u]^{=}& \bF \ar@/^/[u] 
}
\]
This computation appears there in homology as \cite[Proposition~1.7 (b)]{MR2025457}, see also Lewis'
\cite[Theorem~2.1]{MR979507}, where he attributes this to unpublished work of Stong.

\begin{proposition}\label{prop:strMAckHF}
The $RO(\C_2)$-graded Mackey functor $\underline{\HF}^\star$ is represented in Figure~\ref{fig:strucHF}. 
A vertical line represents the product with the Euler class $a\colon S^0 \hookrightarrow S^{\alpha}$, which increases degree by $\alpha$. 
This product induces one of the following Mackey functor maps:
\begin{itemize}
\item[-] the identity between $ \bullet$ functors, 
\item[-] the unique non-trivial morphism $ \underline{\F} \rightarrow \bullet$,
\item[-] the unique non-trivial morphism $ \bullet \hookrightarrow L$.
\end{itemize}

\begin{figure}[h!bt]
 \centering
\definecolor{cqcqcq}{rgb}{0.75,0.75,0.75}
\definecolor{qqqqff}{rgb}{0.33,0.33,0.33}
\begin{tikzpicture}[line cap=round,line join=round,>=triangle 45,x=0.5cm,y=0.5cm]
\draw[->,color=black] (-10,0) -- (10,0);
\foreach \x in {-10,-8,-6,-4,-2,2,4,6,8}
\draw[shift={(\x,0)},color=black] (0pt,2pt) -- (0pt,-2pt) node[below] {\footnotesize $\x$};
\draw[->,color=black] (0,-10) -- (0,10);
\foreach \y in {-10,-8,-6,-4,-2,2,4,6,8}
\draw[shift={(0,\y)},color=black] (2pt,0pt) -- (-2pt,0pt) node[left] {\footnotesize $\y$};
\draw[color=black] (0pt,-10pt) node[right] {\footnotesize $0$};
\clip(-10,-10) rectangle (10,10);
\draw (-0.31,0.5) node[anchor=north west] {$\underline{ \F}$};
\draw (1.7,-1.52) node[anchor=north west] {$L$};
\draw (2.71,-2.52) node[anchor=north west] {$L$};
\draw (3.67,-3.51) node[anchor=north west] {$L$};
\draw (4.65,-4.51) node[anchor=north west] {$L$};
\draw (5.66,-5.49) node[anchor=north west] {$L$};
\draw (6.66,-6.5) node[anchor=north west] {$L$};
\draw (7.65,-7.46) node[anchor=north west] {$L$};
\draw (8.63,-8.47) node[anchor=north west] {$L$};
\draw (9.64,-9.45) node[anchor=north west] {$L$};
\draw (-1.81,1.5) node[anchor=north west] {$\underline{\F}$};
\draw (-12.25,12.46) node[anchor=north west] {$\underline{\F}$};
\draw (-11.24,11.45) node[anchor=north west] {$\underline{\F}$};
\draw (-10.28,10.49) node[anchor=north west] {$\underline{\F}$};
\draw (-9.3,9.51) node[anchor=north west] {$\underline{\F}$};
\draw (-8.29,8.5) node[anchor=north west] {$\underline{\F}$};
\draw (-7.29,7.5) node[anchor=north west] {$\underline{\F}$};
\draw (-6.3,6.51) node[anchor=north west] {$\underline{\F}$};
\draw (-5.32,5.51) node[anchor=north west] {$\underline{\F}$};
\draw (-4.31,4.53) node[anchor=north west] {$\underline{\F}$};
\draw (-3.31,3.52) node[anchor=north west] {$\underline{\F}$};
\draw (-2.4,2.49) node[anchor=north west] {$\underline{\F}$};
\draw (0,0) -- (0,-10);
\draw (2,-2) -- (2,-10);
\draw (3,-3) -- (3,-10);
\draw (4,-4) -- (4,-10);
\draw (5,-5) -- (5,-10);
\draw (6,-6) -- (6,-10);
\draw (7,-7) -- (7,-10);
\draw (8,-8) -- (8,-10);
\draw (9,-9) -- (9,-10);
\draw (10,-10) -- (10,-10);
\draw (-1,1) -- (-1,10);
\draw (-2,2) -- (-2,10);
\draw (-3,3) -- (-3,10);
\draw (-4,4) -- (-4,10);
\draw (-5,5) -- (-5,10);
\draw (-6,6) -- (-6,10);
\draw (-7,7) -- (-7,10);
\draw (-8,8) -- (-8,10);
\draw (-9,9) -- (-9,10);
\draw (-10,10) -- (-10,10);
\draw (-11,11) -- (-11,10);
\draw (9.4,0.9) node[anchor=north west] {$1$};
\draw (0.1,10) node[anchor=north west] {$\alpha$};
\begin{scriptsize}
\fill [color=qqqqff] (-1,2) circle (1.5pt);
\fill [color=qqqqff] (-1,3) circle (1.5pt);
\fill [color=qqqqff] (-1,4) circle (1.5pt);
\fill [color=qqqqff] (-1,5) circle (1.5pt);
\fill [color=qqqqff] (-1,6) circle (1.5pt);
\fill [color=qqqqff] (-1,7) circle (1.5pt);
\fill [color=qqqqff] (-1,8) circle (1.5pt);
\fill [color=qqqqff] (-1,9) circle (1.5pt);
\fill [color=qqqqff] (-1,10) circle (1.5pt);
\fill [color=qqqqff] (0,1) circle (1.5pt);
\fill [color=qqqqff] (0,2) circle (1.5pt);
\fill [color=qqqqff] (0,3) circle (1.5pt);
\fill [color=qqqqff] (0,4) circle (1.5pt);
\fill [color=qqqqff] (0,5) circle (1.5pt);
\fill [color=qqqqff] (0,6) circle (1.5pt);
\fill [color=qqqqff] (0,7) circle (1.5pt);
\fill [color=qqqqff] (0,8) circle (1.5pt);
\fill [color=qqqqff] (0,9) circle (1.5pt);
\fill [color=qqqqff] (0,-10) circle (1.5pt);
\fill [color=qqqqff] (2,-3) circle (1.5pt);
\fill [color=qqqqff] (2,-4) circle (1.5pt);
\fill [color=qqqqff] (2,-5) circle (1.5pt);
\fill [color=qqqqff] (2,-6) circle (1.5pt);
\fill [color=qqqqff] (2,-7) circle (1.5pt);
\fill [color=qqqqff] (2,-8) circle (1.5pt);
\fill [color=qqqqff] (2,-9) circle (1.5pt);
\fill [color=qqqqff] (2,-10) circle (1.5pt);
\fill [color=qqqqff] (3,-4) circle (1.5pt);
\fill [color=qqqqff] (3,-5) circle (1.5pt);
\fill [color=qqqqff] (3,-6) circle (1.5pt);
\fill [color=qqqqff] (3,-7) circle (1.5pt);
\fill [color=qqqqff] (3,-8) circle (1.5pt);
\fill [color=qqqqff] (3,-9) circle (1.5pt);
\fill [color=qqqqff] (3,-10) circle (1.5pt);
\fill [color=qqqqff] (4,-5) circle (1.5pt);
\fill [color=qqqqff] (4,-6) circle (1.5pt);
\fill [color=qqqqff] (4,-7) circle (1.5pt);
\fill [color=qqqqff] (4,-8) circle (1.5pt);
\fill [color=qqqqff] (4,-9) circle (1.5pt);
\fill [color=qqqqff] (4,-10) circle (1.5pt);
\fill [color=qqqqff] (5,-6) circle (1.5pt);
\fill [color=qqqqff] (5,-7) circle (1.5pt);
\fill [color=qqqqff] (5,-8) circle (1.5pt);
\fill [color=qqqqff] (5,-9) circle (1.5pt);
\fill [color=qqqqff] (5,-10) circle (1.5pt);
\fill [color=qqqqff] (6,-7) circle (1.5pt);
\fill [color=qqqqff] (6,-8) circle (1.5pt);
\fill [color=qqqqff] (6,-9) circle (1.5pt);
\fill [color=qqqqff] (6,-10) circle (1.5pt);
\fill [color=qqqqff] (7,-8) circle (1.5pt);
\fill [color=qqqqff] (7,-9) circle (1.5pt);
\fill [color=qqqqff] (7,-10) circle (1.5pt);
\fill [color=qqqqff] (8,-9) circle (1.5pt);
\fill [color=qqqqff] (8,-10) circle (1.5pt);
\fill [color=qqqqff] (9,-10) circle (1.5pt);
\fill [color=qqqqff] (-2,10) circle (1.5pt);
\fill [color=qqqqff] (-2,9) circle (1.5pt);
\fill [color=qqqqff] (-2,8) circle (1.5pt);
\fill [color=qqqqff] (-2,7) circle (1.5pt);
\fill [color=qqqqff] (-2,6) circle (1.5pt);
\fill [color=qqqqff] (-2,5) circle (1.5pt);
\fill [color=qqqqff] (-2,4) circle (1.5pt);
\fill [color=qqqqff] (-2,3) circle (1.5pt);
\fill [color=qqqqff] (-3,10) circle (1.5pt);
\fill [color=qqqqff] (-3,9) circle (1.5pt);
\fill [color=qqqqff] (-3,8) circle (1.5pt);
\fill [color=qqqqff] (-3,7) circle (1.5pt);
\fill [color=qqqqff] (-3,6) circle (1.5pt);
\fill [color=qqqqff] (-3,5) circle (1.5pt);
\fill [color=qqqqff] (-3,4) circle (1.5pt);
\fill [color=qqqqff] (-4,10) circle (1.5pt);
\fill [color=qqqqff] (-4,9) circle (1.5pt);
\fill [color=qqqqff] (-4,8) circle (1.5pt);
\fill [color=qqqqff] (-4,7) circle (1.5pt);
\fill [color=qqqqff] (-4,6) circle (1.5pt);
\fill [color=qqqqff] (-4,5) circle (1.5pt);
\fill [color=qqqqff] (-5,10) circle (1.5pt);
\fill [color=qqqqff] (-5,9) circle (1.5pt);
\fill [color=qqqqff] (-5,8) circle (1.5pt);
\fill [color=qqqqff] (-5,7) circle (1.5pt);
\fill [color=qqqqff] (-5,6) circle (1.5pt);
\fill [color=qqqqff] (-6,10) circle (1.5pt);
\fill [color=qqqqff] (-6,9) circle (1.5pt);
\fill [color=qqqqff] (-6,8) circle (1.5pt);
\fill [color=qqqqff] (-6,7) circle (1.5pt);
\fill [color=qqqqff] (-7,10) circle (1.5pt);
\fill [color=qqqqff] (-7,9) circle (1.5pt);
\fill [color=qqqqff] (-7,8) circle (1.5pt);
\fill [color=qqqqff] (-8,10) circle (1.5pt);
\fill [color=qqqqff] (-8,9) circle (1.5pt);
\fill [color=qqqqff] (-9,10) circle (1.5pt);
\end{scriptsize}
\end{tikzpicture}
\caption{Structure of $\underline{\HF}^\star$}\label{fig:strucHF}
\end{figure}
\end{proposition}

Observe in particular that at each $RO(\C_2)$-degree, the ring $\HF^\star$ is at most one dimensional over $\bF$, 
therefore this ring admits a unique homogeneous basis $h_\star$, as an $RO(\C_2)$-graded vector space,
with $|h_\star|=\star$. In terms of the preferred elements  $a$ and $u$, we have $h_{-n+(n+k)\alpha} = a^ku^{n}$.
If $X$ is a space, or a spectrum, with trivial action, then the non-equivariant $\mathbf Z$-graded cohomology ring 
$\HH\bF^*(X)$ is a subring of $\HF^\star(X)$.
The following proposition explains how this subring, together with the coefficients $\HF^\star$, determines the full $RO(\C_2)$-graded cohomology ring $\HF^\star(X)$, and even more for any equivariant spectrum it determines $\HF^\star(Y \wedge X)$ as a function of $\HF^\star(Y)$.

\begin{proposition}\label{prop:cohoesptriv}
Let $X$ be a $\C_2$-space with trivial action and $Y$ be any $\C_2$-spectrum. There is an isomorphism
\[
 \HF^\star(Y) \otimes_{\mathbf{F}} \HH^\ast(X) \cong \HF^\star(Y \wedge X).
\]
\end{proposition}
\begin{proof} Since the $\C_2$-space $X$ has a trivial action, all maps $\Sigma^nX \rightarrow \HF$ factor through the fixed points of the spectrum 
	$\HF$, giving cohomology classes $[\Sigma^n X,\HH]$, see (\ref{adj:fix}). Thus, we have an isomorphism
	\begin{equation*}
	\HH^*(X) \rightarrow \HF^{* + 0\alpha}(X),
	\end{equation*}
	where $\HF^{* + 0\alpha}(X)$ denotes the graded abelian subgroup of $\HF^{\star}(X)$ consisting in classes whose 
	degree is a multiple of the trivial representation.
	
	The smash product of maps and the pairing  from Proposition~\ref{lem:fix}(2), gives a comparison morphism
	\[
	\HF^\star(Y) \otimes_{\mathbf{F}} \HH^\ast (X) \rightarrow \HF^{\star+ \ast}(Y \wedge X).
	\]
	Fixing the virtual representation $ \star$ and the spectrum $Y$ this is a natural transformation between cohomology theories on spaces. For $x=S^0$, it is clearly the identity on $\HF^\star(Y)$ and
	we conclude that this is an isomorphism for any $X$ from the Eilenberg-Steenrod axioms.
\end{proof}

 When $X$ is a finite $\C_2$-space, one
could alternatively prove the statement  using the dual K\"unneth spectral sequence \cite[Theorem~1.6]{MR2205726} studied by Lewis and Mandell.

%

Here is an example of a computation with the geometric fixed points functor; it is certainly a folklore result but since we need it explicitly 
we provide a complete proof, which is a mod $2$ analog of the proof of \cite[Proposition~3.18]{MR3505179}. This splitting also appears in
Behrens and Wilson's \cite[Section~2]{MR3856165}.

\begin{proposition}\label{prop:geofixehf}
There is an $\HH$-linear splitting of the geometric fixed points of the Eilenberg-Mac Lane spectrum $\HF$:
\[
\HH[b] := \Phi^{\C_2}(\HF) \simeq \bigvee_{k \in \mathbb{N}} S^k \wedge \HH.
\]
\end{proposition}
\begin{proof}
Proposition~\ref{lem:fix}(1) identifies the fixed points of $\HF$ as $\iota \HH$ and we choose a cellular model to make  sure
that $\Phi^{\C_2} (\iota \HH) \simeq \HH$ by \cite[Proposition~B.182]{MR3505179}.
The inclusion of fixed points $\iota \HH \rightarrow \HF$ can then be seen as a map of (right) $\iota \HH$-modules,
hence becomes after applying $\Phi^{\C_2}$ a map of $\HH$-modules.

Robinson, \cite{MR910828}, proved that $\HH$-modules split as products of Eilenberg-Mac Lane spectra. It is thus enough to compute 
the homotopy groups $\pi_\ast(\Phi^{\C_2}(\HF))$ to identify the homotopy type of $\Phi^{\C_2}(\HF)$. Let $k$ be an integer. 
 We use first the adjunction(\ref{adj:fix}), 
\[
	\pi_k(\Phi^{\C_2}(\HF))  =  [S^k, (\widetilde{E\C_2} \wedge \HF)^{\C_2}]  \cong [\iota S^k, \widetilde{E\C_2} \wedge \HF]^{\C_2}  
	= [\iota S^k, \colim_{n} S^{n\alpha}\wedge \HF ]^{\C_2}.
\]	
We then pull out the colimit by compactness and get	
\[
\colim_n[\iota S^k,S^{n\alpha}\wedge \HF ]^{\C_2}  \cong  \colim_i[S^{k-(k+i)\alpha}, \HF ]^{\C_2}  \cong \bF\{b^k\} \cong \F  \textrm{ for $k \geq 0$},
\]
 where we observe that all morphisms in the colimit are isomorphisms for $i\geq 0$ and 
the last identification follows from the computation of $\HF^\star$ in Proposition~\ref{prop:strMAckHF}.
The meaning of the symbol $b^k$ is explained in the next remark.
\end{proof}

\begin{remark}\label{rem:geofixehf}
A more complete computation, as carried out in \cite[Section~2]{MR3856165} for instance, shows that 
\[
\pi^{\C_2}_\star(\Phi^{\C_2}(\HF)) \cong \F[a^{\pm 1},u].
\]
In integral degree we have thus $\pi_\ast(\Phi^{\C_2}(\HF)) \cong \F[a^{-1}u]$, and here the degree is homological, hence $a^{-1} u$ has degree one.
The element $b^k$ in our proof stands for $(a^{-1}u)^k$, and is, as expected, a polynomial generator.

We record also from the above computation and for future use that, for any $i \geq 0$, the class 
$h_{(k+i)\alpha -k} = a^{i}u^{k} \in \HF^{-k+(k+i)\alpha}$
corresponds to the summand in degree $k$ in the wedge decomposition $\Phi^{\C_2}(\HF) \simeq \bigvee_{k \in \mathbb{N}} S^k \wedge \HH$
in the sense that when we consider $a^{i}u^k$ as a map $S^{k-(k+i)\alpha} \rightarrow \HF$, the geometric fixed points 
$\Phi^{\C_2}(a^{i}u^{k})\colon S^k \rightarrow \HH[b]$ is the non-trivial map that picks up the component in degree~$k$.
\end{remark}

\section{Homologically pure equivariant spaces}\label{sec:pureimpliesconj}
Inspired by the definition of purity given by Hill, Hopkins, and Ravenel in \cite[Definition~4.56]{MR3505179}, we define homologically
pure $\C_2$-equivariant spaces in this section and proceed to show that they are conjugation spaces. This means that we have to
explain how to construct a conjugation frame out of the equivariant data at hand.

\subsection{Definition and basic properties}\label{subsec:def}

 Recall that a space $X$ is of \emph{finite type} if its cohomology $\HH^\ast(X)$ is of finite type.

\begin{definition}\label{def:homolpurity}
An equivariant space is \emph{homologically pure} if there exist a set $I$, natural numbers $n_i$ for any $i \in I$, and
a weak equivalence of $\HF$-modules
\[
  X \wedge \HF\simeq \bigvee_{i \in I}\Sigma^{n_i(1+\alpha)}\HF.
\]
 We say moreover that $X$ is of \emph{finite type} if the set $\{ i \in I \ | \ n_i =n \}$ is finite for each natural number~$n$.
\end{definition}

In other words $ X \wedge \HF$ splits as a wedge of pure slices,  and if of finite type, there are finitely many in each degree. 
For a more thorough discussion of freeness  and finite type see~\cite[Definition~2.27]{HillPreprint}.
Observe that a homologically pure space is a space such that the free $\HF$-module it generates is built out of conjugation spheres.

\begin{example}\label{ex:spherical}
Spherical conjugation spaces as introduced in \cite[Section~5.2]{MR2171799} are examples of pure spaces. 
Indeed, if $X$ is a spherical conjugation space, one can consider the spectral sequence 
associated with the cellular decomposition of $X$, that is the skeletal filtration, which computes equivariant homology. 
This is the spectral sequence considered 
in \cite[2.24]{HK}. It collapses for degree reasons, since the homology of a point is zero in the quadrant $n + m \alpha$, for $n,m > 0$, see
Proposition~\ref{prop:strMAckHF}. 
Consequently, $(X \wedge \HF)_{\star}$ is free over $\HF_{\star}$, and the classical argument from \cite[Lemma~II.11.1]{MR1324104} shows that
the spectrum $X \wedge \HF$ splits as a wedge of suspensions of $\HF$ by multiples of the regular representation.

Alternatively, when $X$ is finite, one could use first the very general splitting result \cite[Theorem~6.13]{Clover} and identify
afterwards the bidegrees of the corresponding representation spheres.
\end{example}

Let us first list basic properties of homologically pure spaces. 

\begin{lemma}\label{lem:relequivnonequiv}
	Let $X$ be a homologically pure space of finite type,  with $X \wedge \HF \simeq \bigvee_{i \in I}\Sigma^{n_i(1+\alpha)}\HF$.  Then
	\begin{enumerate}
		\item for any $\star\in RO(\C_2)$ we have an equivalence of function spectra:
		\[
		F_{\C_2}(S^{-\star}\wedge X,\HF) \simeq \bigvee_{i \in I}F_{\C_2}(S^{-\star+n_i(1+\alpha)},\HF)
		\]
		\item the ordinary mod $2$ cohomology of the space $X^u$ is concentrated in even degrees~$2n_i$;
		\item we have an isomorphism $\HF^\star (X) \cong \bigoplus_{i \in I} \HF^\star\{x_i\}$, where $|x_i| = n_i(1 + \alpha)$;
		\item the restriction map in the Mackey functor structure  induces an isomorphism $\HF^{\ast(1+\alpha)}(X) \rightarrow \HH^{2\ast}(X^u)$.
	\end{enumerate}
\end{lemma}
\begin{proof}
 (1) Since $X$ is homologically pure and $\HF$ is a ring spectrum, we use first the free-forgetful adjunction:
\[
F_{\C_2}(S^{-\star} \wedge X,\HF)  \simeq  F_{\HF-mod}(\Sigma^{-\star} \HF \wedge X,\HF) \simeq 
F_{\HF-mod}(\bigvee_{i \in I }S^{-\star + n_i(1+\alpha)} \wedge \HF,\HF)
\]
We go back via the free-forgetful adjunction and get:
\[
F_{\C_2}(\bigvee_{i \in I}S^{-\star+n_i(1+\alpha)},\HF)
 \simeq \prod_{i \in I}F_{\C_2}(S^{-\star+n_i(1+\alpha)},\HF)
\]
The finite type assumption finally allows us to identify the product with a wedge $\bigvee_{i \in I}F_{\C_2}(S^{-\star+n_i(1+\alpha)},\HF)$
since the homotopy groups of the summands are non-trivial only for finitely many of them in any given bidegree. 


(2) The forgetful functor is monoidal, therefore the splitting of a homologically pure space yields in particular a splitting
$X^u \wedge \HH \simeq  \bigvee_{i \in I}\Sigma^{2n_i}\HH$. In particular the ordinary homology, and hence the cohomology are concentrated in
even degrees $2n_i$.

(3) Follows from (1) since 	$\HF^\star (X) \cong \pi_{-\star} F_{\C_2}(X,\HF)$.

(4) Restriction is given by forgetting the action as we have seen in  Lemma~\ref{lemma:restriction} and the splitting used in (2) yields
a commutative diagram
	\[
	\xymatrix{
		\HF^{\star}(X) \ar[r]^-\sim  \ar[d]^\rho & \bigoplus_{i \in I} \HF^{\star}(S^{n_i(1+\alpha)}) \ar[d]^\rho \\
		\HH^{|\star|}(X^u) \ar[r]^-\sim  & \bigoplus_{i \in I} \HH^{|\star|}(S^{2n_i}).
	}
	\]
Let us concentrate on degrees of the form $n_i(1+\alpha)$. The top right part is a direct sum
of copies of $\HF^{n_i(1+\alpha)}(S^{n_i(1+\alpha)}) \cong \HF^0(S^0) \cong \bF$ and the bottom right part is 
$\HH^{2n_i}(S^{2n_i}) \cong \HH^0(S^0) \cong \bF$. The vertical arrow is the restriction and is an isomorphism by definition of 
the Mackey functor $\underline{\F}$.
\end{proof}

Let us stress that the key structural property of conjugation spaces is $(4)$. As we will see, the section $\sigma$ and the various maps $\kappa_i$ 
are defined for any $\C_2$-space, with the caveat that their natural source is an equivariant cohomology group $\HF^{n(\alpha+1)}(X)$ and not ordinary cohomology. 
It is only through the isomorphism $(4)$ that these maps descend to $\HH^{2\ast}(X^u)$.

\subsection{The degree halving isomorphism of a homologically pure space}\label{sec:kappa}
The first ingredient in a cohomology frame is the degree halving isomorphism $\kappa_0$ between
the cohomology of the $\C_2$-equivariant space and that of the fixed points.
It is convenient to define first a global map $\kappa_T$, where the letter T stands for ``total'',  that encodes the conjugation equation.
The halving isomorphism $\kappa_0$ appearing in Definition~\ref{def:conjspaceHHP} will be recovered from~$\kappa_T$.

\begin{definition}\label{def:tildekappa}
Let $X$ be a $\C_2$-space. Then the inclusion of fixed points $X^{\C_2} \rightarrow X$ induces 
$\kappa_T\colon \HF^{*(1+\alpha)}(X) \longrightarrow  \HF^{*(1+\alpha)}(X^{\C_2})$.
\end{definition}

If $X$ is homologically pure, we can precompose this map by the isomorphism from Lemma~\ref{lem:relequivnonequiv} (4)
to obtain a map we denote still by $\kappa_T$:
\[
\HH^{2\ast}(X^u) \xrightarrow{\rho^{-1}} \HF^{*(1+\alpha)}(X) \stackrel{\kappa_T}{\longrightarrow}  \HF^{*(1+\alpha)}(X^{\C_2}).
\]
Remember from Proposition~\ref{prop:cohoesptriv} that $\HF^{\star}(X^{\C_2}) \cong \HF^{\star} \otimes \HH^\ast(X^{\C_2})$.
In particular, if we start with a class $x \in \HF^{n(1+\alpha)}(X)$, we obtain from the structure of the coefficient ring $\HF^{\star}$
a decomposition of $\kappa_T(x)$ as $\sum_{i=0}^{n} a^{n-i}u^{i} \otimes \kappa_i^e(x)$. Recall that $a^{n-i}u^{i}$ belongs to $\HF^{n\alpha-i}$.

\begin{definition}\label{def:tildekappai}
Let $X$ be a $\C_2$-space and $x \in \HF^{n(1+\alpha)}(X)$. Then
the \emph{equivariant $\kappa$-classes} are the elements $\kappa_i^e(x) \in \HH^{n+i}(X^{\C_2})$. 
\end{definition}

Before we show that the degree halving isomorphism $\kappa_0$ can be chosen to be $\kappa_0^e$  we need a small lemma:


\begin{lemma}\label{lem:equivptsfixes}
Let $X$ be any $\C_2$-space, then the inclusion of fixed points $X^{\C_2} \rightarrow X$ induces an equivalence of equivariant spectra
\[
\widetilde{E\C_2} \wedge \Sigma^{\infty}(X^{\C_2}) \rightarrow \widetilde{E\C_2} \wedge \Sigma^{\infty}_{\C_2}X.
\]
\end{lemma}
\begin{proof}
Observe firstly that the underlying non-equivariant spectra on both sides are contractible by Remark~\ref{rem:ring}. 
Secondly, taking fixed points on both sides coincides by definition with taking geometric fixed points of $\Sigma^{\infty}(X^{\C_2})$ 
and $\Sigma^{\infty}_{\C_2}X$ respectively. We conclude
by Proposition~\ref{prop:propgeofixedpoints} (4) that they agree with $\Sigma^{\infty} (X^{\C_2})$.
\end{proof}

For any $\C_2$-space $X$ denote by $\kappa$ the following composite map, where in step 3 and 5 we use the tensor-hom adjunction for 
modules over the ring spectrum $\widetilde{E\C_2}$, see Remark~\ref{rem:ring} and the first map is induced by $S^0 \rightarrow \widetilde{E\C_2}$:
\begin{align*}
\HF^{n(1+\alpha)}(X)  & =  [S^{-n(1+\alpha)}\wedge X,\HF]^{\C_2}\\
 & \rightarrow  [S^{-n(1+\alpha)}\wedge X,\widetilde{E\C_2}\wedge \HF]^{\C_2}\\
& \cong  [\widetilde{E\C_2} \wedge S^{-n(1+\alpha)}\wedge X,\widetilde{E\C_2}\wedge \HF]^{\C_2} \\
& \cong  [\widetilde{E\C_2} \wedge (S^{-n(1+\alpha)}\wedge X)^{\C_2},\widetilde{E\C_2}\wedge \HF]^{\C_2} \\
\shortintertext{by Lemma~\ref{lem:equivptsfixes}}
& \cong  [ \iota(S^{-n(1+\alpha)}\wedge X)^{\C_2},\widetilde{E\C_2}\wedge \HF]^{\C_2} \\
& \cong  [ (S^{-n(1+\alpha)}\wedge X)^{\C_2},(\widetilde{E\C_2}\wedge \HF)^{\C_2}]\\
\shortintertext{by the $\iota \dashv (-)^{\C_2}$ adjunction}
& \cong  [ S^{-n}\wedge X^{\C_2},(\widetilde{E\C_2}\wedge \HF)^{\C_2}] \\
& \cong  [ S^{-n}\wedge X^{\C_2},\bigvee_{k \in \mathbb{N}} S^k \wedge  \HH]\\  
\shortintertext{by Proposition~\ref{prop:geofixehf}, and then, as $\bigvee_{k \in \mathbb{N}} S^k  \hookrightarrow \prod_{k \in \mathbb{N}} S^k$ is an equivalence}
& \cong  [ S^{-n}\wedge X^{\C_2},\prod_{k \in \mathbb{N}} S^k \wedge  \HH]\\ 
& \cong   \prod_{k \in \mathbb{N}}[ S^{-n}\wedge X^{\C_2}, S^k \wedge  \HH] \twoheadrightarrow  \HH^n(X^{\C_2}) 
\end{align*}
where we finally project on the factor $k=0$.
To sum up, the map $\kappa$ consists in representing a cohomology class whose degree is a multiple of the regular representation
by a map $S^{-n(1+\alpha)}\wedge X \rightarrow \HF$ and taking then its geometric fixed points. Since $X$ is a space, 
$\Phi^{\C_2}(S^{-n(1+\alpha)}\wedge X) \simeq S^{-n}\wedge X^{\C_2}$ and the splitting of $\Phi^{\C_2} \HF$ into a wedge of
suspended copies of $\HH$ allows us to project onto one factor, the zeroth one in this case. More generally:

\begin{lemma}\label{lem:kappanadPhi}
Let $pr_i \colon \bigvee_{b \in \mathbb{N}} S^b \wedge H \rightarrow S^i \wedge H$ denote the projection on the $i$-th wedge summand.
For any $x \in \HF^{n(1+\alpha)}(X)$, the cohomology class $\kappa_i^e(x) \in \HH^{n+i}(X^{\C_2})$ coincides with $pr_i \circ \Phi^{\C_2}(x)$.
In particular the maps $\kappa_0^e$ and $\kappa$ above coincide.
\end{lemma}

\begin{proof}
If we represent $x$ by a map $S^{-n(1+ \alpha)} \wedge X \rightarrow \HF$, then $\kappa_T(x)$ is obtained by
precomposing with the fixed point inclusion $X^{\C_2} \hookrightarrow X$. We have already seen that this map
decomposes as a sum of products of classes $\kappa_i^e(x) \in \HH^{n+i}(X^{\C_2})$ with $a^{n-i}u^{i}$. In other
words $\kappa_T(x)$ can be written as a sum of classes of maps
\[
\xymatrix{
S^{-n(1+\alpha)} \wedge X^{\C_2} \simeq S^{i-n\alpha} \wedge S^{-n-i} \wedge X^{\C_2} \ar[rrr]^-{a^{n-i}u^{i} \wedge \kappa_i^e(x)} 
& & & \HF \wedge \HH \ar[r]^-{\mu} &  \HF
}
\]
We apply now geometric fixed points. Since $\Phi^{\C_2}$ is additive and monoidal, and from Lemma~\ref{lem:equivptsfixes}, we get a sum of maps
\[
\xymatrix{
S^{-n} \wedge X^{\C_2} \simeq S^{i} \wedge S^{-n-i} \wedge X^{\C_2} \ar[rrr]^-{b^i \wedge \kappa_i^e(x)} 
& & & \HH[b] \wedge \HH \ar[r]^-{\mu} &  \HH[b]
}
\]
The inclusion $\HH \rightarrow \HF$ induces on geometric fixed points the bottom summand inclusion $\HH \rightarrow \HH[b]$ by
$\HH$-linearity, see Proposition~\ref{prop:geofixehf}.
Thus projection on the $j$-th summand of $\mu \circ (b^i \wedge \kappa_i^e(x))$ is trivial for $j \neq i$ and $\kappa_i^e(x)$ when $j=i$.
\end{proof}

We finally come back to homologically pure spaces. Just like with the map $\kappa_T$, we still denote by 
$\kappa_0^e \colon \HH^{2\ast}(X) \rightarrow \HH^\ast(X^{\C_2})$ the precomposition of $\kappa_0^e$ by the  isomorphism 
$\HH^{2\ast}(X) \stackrel{\cong}{\rightarrow} \HF^{\ast(1+\alpha)}(X)$ when $X$ is pure.

\begin{proposition}\label{prop:kappaisohompur}
Let $X$ be a homologically pure space  of finite type. Then the morphism $\kappa_0^e\colon  \HH^{2\ast}(X) \rightarrow \HH^{\ast}(X^{\C_2})$ is 
an isomorphism.
\end{proposition}
\begin{proof}
Because the ``$a$-periodization''  and the ``free $\HF$-module'' functors commute, see Lemma~\ref{lem:aperprop},
the same computation as in Lemma~\ref{lem:relequivnonequiv}(1) shows that we have an equivalence of function spectra:
\[
F_{\C_2}(S^{-\star}\wedge X ,\widetilde{E\C_2} \wedge \HF) \simeq 	\bigvee_{i \in I} F_{\C_2}(S^{-\star}\wedge S^{n_i(1+\alpha)} ,\widetilde{E\C_2} \wedge \HF) 
\]
As a consequence, fixing an integer $m \in \mathbb{N}$, the following diagram defining $\kappa^e_0$ for $X$ and for a wedge of spheres commutes:
\[
\xymatrix{
		\HF^{2m}(X) \ar[r]^-\cong &  [S^{-m(1+\alpha)}\wedge X,\HF]^{\C_2} \ar[d]^{\cong} \ar[r]  &  [S^{-m(1+\alpha)}\wedge X,\widetilde{E\C_2} \wedge \HF]^{\C_2} \ar[d]^{\cong} \ar[r] &  \cdots \\
		 &  \bigoplus_{i \in I}[S^{(n_i-m)(1+\alpha)},\HF]^{\C_2} \ar[r] & \bigoplus_{i \in I}[S^{(n_i-m)(1+\alpha)},\widetilde{E\C_2} \wedge \HF]^{\C_2}  \ar[r] & \cdots \\
	\cdots \ar[r]	 &  [ S^{-m}\wedge X^{\C_2},  \HH[b]] \ar@{->>}[r] \ar[d]^{\cong} & 
		 [ S^{-m}\wedge X^{\C_2}, \HH] \ar[d]^\cong \\
	\cdots	 \ar[r] & \bigoplus_{i \in I}[S^{(n_i-m)},\HH[b]] \ar@{->>}[r] & \bigoplus_{i \in I}[ S^{(n_i-m)}, \HH]
	}
	\]
Because the definition of the map $\kappa^e_0$ only involves adjunctions and operations on the second variable in the homotopy classes of maps, 
it is immediate that it is an additive map; in other words the bottom line above respects the direct sum decomposition. Hence it is enough to show the 
proposition for a single conjugation sphere $S^{n(1+\alpha)}$, in which case the only non-trivial statement is when $m=n$, i.e. $S^{(n-m)}= S^0$. 

Let us examine the definition of $\kappa_0^e$ in this case. By the computation of the coefficient ring $\HF^{\star}$ we have a generator 
in degree zero that is divisible by $a$ (see Proposition~\ref{prop:strMAckHF}), 
and the first arrow is an isomorphism. So is the second one because $\HH$ is an Eilenberg-Mac Lane spectrum. 
\end{proof}

\subsection{The section of a homologically pure space}
\label{subsec:section}
This map is comparatively easier to define. We start with a short computation.

\begin{lemma}\label{lemma:EC}
The cohomology of the universal $\C_2$-space is $\HF^{\star}(E\C_{2+}) \cong \F[a, u^{\pm 1}]$. 
As a consequence, for any $\C_{2}$-space $X$, $\HF^{\star}(E\C_{2+} \wedge X) \cong H^{|\star|}(E\C_{2+} \wedge_{\C_2} X)$.
\end{lemma}

\begin{proof} To compute the cohomology of $E\C_{2+} \simeq \colim_n S(n\alpha)_+$, we first compute the cohomology of the spheres $S(n\alpha_+)$. An easy induction argument on $n$ using the cofibration sequence:
		\[
		S(n\alpha) \rightarrow S((n+1)\alpha) \rightarrow \C_{2+} \wedge S^{n} 
		\]
	shows that $\HF^{\star}(S(n\alpha_+)) \cong \F[a,u^{\pm 1}]/a^{n}$ (see \cite[Section 3]{Clover} for details). The Milnor sequence gives then the result.
	
	For the second part, observe that the homotopical uniqueness of the space $E\C_{2+}$ implies that, by smashing and using the ring structure of $\HF$,  $\HF^{\star}(E\C_{2+} \wedge X)$ is and $\HF(E\C_{2+})$-module. In particular multiplication by $u^{\pm1}$ is an isomorphism in $\HF^{\star}(E\C_{2+} \wedge X)$. Hence multiplication by $u^{-m}$ induces an isomorphism:
	\[
	\HF^{n +m\alpha}(E\C_{2+} \wedge X) \cong \HF^{n +m}(E\C_{2+} \wedge X)
	\]
	and it is enough to show the result in integral grading. This is now a particular case of \cite[Chap II. Thm 8.1]{LMS86}. Briefly, one the one hand, the restriction map of the Mackey functor $\underline{\bF}$ is an isomorphism. In particular, the morphism $\iota \colon \iota \HH \rightarrow \HF$, adjoint
to the identity map on $\HH$ (remember that $(\HF)^{\C_2} \simeq \HH$ by Proposition~\ref{lem:fix}(1)), is a non-equivariant weak equivalence. 
By \cite[Lemma 0.4]{MR1230773} this property is equivalent to $\HF$ being a split spectrum, which is well known, and used for 
instance in \cite[proof of Proposition 6.2]{HK}.

On the other hand, (the suspension spectrum of) $E\C_{2+}\wedge X$ is the paradigm of a free $\C_2$-spectrum. Because of this last fact:
\[
\HH^\ast(E\C_{2+} \wedge_{\C_2} X) = [S^{-*} \wedge(E\C_{2+} \wedge X)_{\C_2},\HH] \cong [S^{-*} \wedge E\C_{2+} \wedge X,\iota\HH]^{\C_2}.
\]
We can compose further with the map $\iota$
\[
[S^{-*} \wedge E\C_{2+} \wedge X, \iota\HH]^{\C_{2}} \rightarrow [S^{-*} \wedge E\C_{2+} \wedge X,\HF]^{\C_2}.
\]
which is an isomorphism since $\iota$ is an underlying equivalence and therefore induces an equivalence of function spectra 
$F_{\C_2}(E\C_{2+},\iota)$ by \cite[Proposition 1.1]{MR1230773}. This concludes the proof.
\end{proof}

 When $X$ is a finite $\C_2$-space, the last computation is more direct, see for example Hazel's \cite[Lemma~3.1]{Hazel19}.
We are now ready to define the section of the cohomology frame.

\begin{definition}\label{def:tildesigma}
Let $X$ be a $\C_2$-space. The equivariant map $E\C_{2+} \wedge X \rightarrow S^0 \wedge X$ 
induces a \emph{global section} $\sigma_T\colon \HF^{*(1+\alpha)}(X) \longrightarrow  \HF^{*(1+\alpha)}(E\C_{2+} \wedge X)$.
\end{definition}

When $X$ is a homologically pure space, Lemma~\ref{lem:relequivnonequiv}.(4) allows us to identify the source with
ordinary cohomology and Lemma~\ref{lemma:EC} identifies the target with Borel cohomology.

\begin{definition}\label{def:sigma}
Let $X$ be a homologically pure space. The \emph{section} $\sigma$ is the composite:
\[
\HH^{2*}(X^u) \cong \HF^{*(1+\alpha)}(X)  \stackrel{\sigma_T}{\longrightarrow} \HF^{*(1+\alpha)}(E\C_{2+} \wedge X) \cong H^{2*}(E\C_{2+} \wedge_{\C_2} X)
\]
\end{definition}

\begin{lemma} \label{lem:sigmais section}
Let $X$ be a homologically pure space. The morphism $\sigma$ is a section of the restriction
$\rho \colon H^*(E\C_{2+} \wedge_{\C_2} X) \rightarrow H^*(X)$.
\end{lemma}
\begin{proof}
The morphism $\sigma$ is induced by $E\C_{2+} \wedge X \rightarrow X$,  up to natural isomorphisms. 
This map induces a morphism of \emph{Mackey functors} $\HF^{\star}(X) \rightarrow \HF^{\star}(E\C_{2+} \wedge X)$,
and in particular, it is compatible with the restriction morphisms. Consequently, the following square commutes:
\[
\xymatrix{  
\HF^{*(1+\alpha)}(X) \ar[r] \ar[d]_{\rho}^{\cong}  & \HF^{*(1+\alpha)}(E\C_{2+} \wedge X) \ar[d]^{\rho} \\ 
H^{2*}(X^u) \ar[r]^-{\cong}& H^{2*}((E\C_{2+} \wedge X)^u) 
}
\]
where the vertical maps are the restrictions of the corresponding Mackey functors, identified in Lemma~\ref{lemma:restriction}. 
The left hand side one is the isomorphism given by Lemma~\ref{lem:relequivnonequiv}.(4) and used in the definition of $\sigma$,
the bottom map is an isomorphism since $(E\C_2)^u \simeq pt$.
\end{proof}

\subsection{The cohomology frame of a homologically pure space}
We are finally ready to prove the main result of this section. As we already have constructed
a halving isomorphism and a section of the restriction map, it only remains to check the conjugation equation.

\begin{theorem}\label{prop:homolpureisconj}
Let $X$ be a homologically pure space of finite type, then it is a conjugation space of finite type.
\end{theorem}

\begin{proof} Observe that trivially a homologically pure space of finite type $X$ has ordinary cohomology $\HH^\ast(X)$ of finite type.
We must show that the maps $\kappa^e_0$ and $\sigma$ defined above satisfy the conjugation equation.
The commutative square~:
	\[
	\xymatrix{
	E\C_{2+} \wedge X^{\C_2}  \ar[r] \ar[d] & X^{\C_2} \ar[d]\\
	E\C_{2+} \wedge X  \ar[r]  & X \\
}
	\]
	induces a commutative square:
	\[
	\xymatrix{
		  \HF^{\ast(1+\alpha)}(X) \ar[r]^-{\sigma} \ar[d]_-{\kappa_T} &  \HF^{\ast(1+\alpha)}(E\C_{2+} \wedge X) \ar[d] \\
	\HF^{\ast(1+\alpha)}(X^{\C_2}) \ar[r]  &  \HF^{\ast(1+\alpha)}(E\C_{2+} \wedge X^{\C_2}). \\
	}
	\]
To understand where the conjugation equation comes from it is enough to understand the bottom line.
 Since $X$ is pure of finite type, the fixed points space $X^{\C_2}$ has finite type, so that
Proposition~\ref{prop:cohoesptriv} applies: The target coincides with the degree $*(1+\alpha)$ part of $\HF^\star(E\C_{2+}) \otimes H^*(X^{\C_2})$.
The decomposition of $\kappa_T(x)$  in terms of $\kappa^e_i(x)$'s is thus preserved by this map. 

But since $\HF^{\star}(E\C_{2+}) \cong \pi_{-\star}(F_{\C_2}(E\C_{2+},\HF))$
the computation that has been used in Proposition~\ref{prop:algstrHF} shows that 
$\kappa_T(x) = \sum_{i=0}^{n} a^{n-i}u^{i} \otimes \kappa_i^e(x)$ is sent
to the corresponding sum in $\HF^{\star}(E\C_{2+}) \otimes H^*(X^{\C_2})$. The identification with the non-equivariant cohomology is obtained
by $u$-periodicity, as mentioned in the introduction of Subsection~\ref{subsec:section}. Hence we push this element by multiplying by $u^{-n}$
so as to have a sum $\sum_{i=0}^{n} (au^{-1})^{n-i} \otimes \kappa_i^e(x)$, where the degree one element $a u^{-1}$ identifies with the only 
non zero element in $H^1(B\C_{2+})$, i.e. the polynomial generator $b$. This concludes the proof.
\end{proof}

\section{Conjugation spaces are pure}\label{sec:conjimpliespure}
We prove in this section that any conjugation space is homologically pure. This provides the characterization of conjugation spaces in terms of
purity, our main contribution in this paper. We start with a lemma.

\begin{lemma}\label{lem:lift}
Let $X$ be a conjugation space. For any class $x \in \HH^{2n} (X^u)$, there is a class $\tilde x \in \HF^{n(1+\alpha)} (X)$
such that the restriction $\rho(\tilde x) = u^{n} x \in \underline{\HF}_{\C_2}^{\star}(X) \cong H^*(X^u) [u^{\pm 1}]$.
\end{lemma}
\begin{proof} 
The isotropy separation and the inclusion $X^{\C_2} \rightarrow X$ provide a commutative diagram
\[
\xymatrix
{ 
X^{\C_2} \ar[d] & \ar[d] \ar[l] E{\C_2}_+ \wedge X^{\C_2} & \ar[d] \ar[l] \Sigma^{-1} \widetilde{E{\C_2}} \wedge X^{\C_2} \\
X & \ar[l] E{\C_2}_+ \wedge X & \ar[l] \Sigma^{-1} \widetilde{E{\C_2}} \wedge X  
}
\]
In terms of Mackey functors, this yields a commutative diagram of $RO(\C_2)$-graded groups:
\[
\xymatrix@C=-0,1cm
{
\HH^*(X^{\C_2})[u^{\pm 1}] \ar[rr]^{\cong} \ar@/_/[dr]_{\tr} && \HH^*(X^{\C_2})[u^{\pm 1}] \ar@/_/[dr]_{\tr} \ar[rr]  && 0 \ar@/_/[dr]  \\
	& \HF^{\star}(X^{\C_2}) \ar[rr] \ar@/_/[ul]_{\rho} &&\HH^{\ast}(X^{\C_2})[a,u^{\pm1}] \ar@/_/[ul]_{\rho}  \ar[rr]  && **[r]{\HF^{\star-1}
	(\widetilde{E\C_2} \wedge X^{\C_2})} \ar@/_/[ul] \\
\HH^*(X)[u^{\pm 1}] {\ar[rr]^{\quad \quad \cong} |!{[d];[dr]}\hole} \ar@/_/[dr]_{\tr} \ar[uu]  && \HH^*(X)[u^{\pm 1}] \ar@/_/[dr]_{\tr}    
\ar[rr] |!{[dl];[dr]}\hole \ar[uu] |!{[ul];[ur]}\hole  && 0 \ar@/_/[dr] \ar[uu] |!{[ul];[ur]}\hole  \\
	& \HF^{\star}(X) \ar@/_/[ul]_{\rho}  \ar[rr] \ar[uu]  && \HF^{\star}(E\C_{2+} \wedge X)  \ar@/_/[ul]_{\rho} \ar[rr]_{\delta} \ar[uu]_(0.3){r_{\C_2}} && **[r]
	{\HF^{\star-1}(\widetilde{E\C_2} \wedge X)} \ar@/_/[ul] \ar[uu]^=  \\   
}
\]
where the front face corresponds to the evaluation at $\C_2$. As for the back face, corresponding to the evaluation at the trivial subgroup, 
the two zero entries come from the fact that  $\widetilde{E\C_2}^u$ is contractible and on the left hand side we incorporated the identification 
from Lemma~\ref{lemma:restriction}, which explains also the isomorphism between the middle and left entries, since there is also a non-pictured 
zero on the left of the diagram. 

The last group we have to identify is $\HF^\star(E\C_{2+} \wedge X^{\C_2})$, sitting in the middle of the top row in the front face. 
We first use Proposition~\ref{prop:cohoesptriv} to identify it with  $\HF^\star(E\C_{2+}) \otimes \HH^*(X^{\C_2})$. We obtain the claimed isomorphism
with $\HH^{\ast}(X^{\C_2})[a,u^{\pm1}]$ since $\HF^\star(E\C_{2+}) \cong \F[a, u^{\pm 1}]$ by Lemma~\ref{lemma:EC}.
		
For any element $x \in \HH^{2n} (X^u)$, consider the element $u^{n}x \in \HH^{n(1+\alpha)}(X)[u^{\pm 1}]$ in the back 
lower left corner of the diagram.
We will exhibit a lift $\tilde{x} \in \HF^{n(1+\alpha)}(X)$ of this element (that is an element $\tilde{x}$ such that $\rho(\tilde{x})= u^{n}x$). \\
		
By hypothesis, $X$ being a conjugation space, the restriction 
$\rho \colon \HF^{\ast}(E\C_{2+} \wedge X) \cong \HH^*(E\C_{2+}\wedge_{\C_2}X) \rightarrow \HH^*(X)$ is surjective
in integral grading, where the isomorphism has been established in Lemma~\ref{lemma:EC}.
Hence, by $u$-periodicity, surjectivity holds true in arbitrary $RO(\C_2)$-grading. 
In particular  $u^{n}x$ admits a lift $u^{n}s(x)$ in $\HF^{n(1+\alpha)}(E\C_{2+} \wedge X)$.
		
Now, by commutativity of the front diagram, $\delta$ is the composite 
\[
\HF^{\star}(E\C_{2+} \wedge X) \stackrel{r_{\C_2}}{\longrightarrow} \HF^{\star}(E\C_{2+}) \otimes H\F^*(X^{\C_2}) \rightarrow 
\HF^{\star-1}(\widetilde{E\C_2} \wedge X^{\C_2}),
\]
where, by construction, the first map is given by the conjugation equation: 
\[
r_{\C_2} [u^{n} s (x)] = u^{n} (\sum \kappa_i(x) (a u^{-1})^{n-i}) = a^n \kappa_0(x) + \sum_{i=1}^n \kappa_i(x) u^{i}a^{n-i}
\]
As observed in the proof of Theorem~\ref{prop:homolpureisconj}, the polynomial generator denoted by $u$ in \cite{MR2171799} is $au^{-1}$ here. The sum above defines actually an element
in $\HF^{\star} \otimes H\F^*(X^{\C_2})$, hence is sent to zero in $\HF^{n(\alpha) +n-1}(\widetilde{E\C_2} \wedge X)$.

Thus, $\delta( u^{n}s(x)) = 0$ and there is a lift $\tilde{x} \in \HF^{n(1+\alpha)}(X)$ of the element $u^{n}s(x)$. 
By naturality, $\rho(\tilde{x})$ is identified with $\rho[u^{n}s(x)] = u^{n} x$ since $s$ is a section of $\rho$.
\end{proof}
		
The above proof gives in fact a bit more.		
\begin{proposition}\label{prop:kappas}
Let $X$ be a conjugation space and $x \in \HH^{2n} (X^u)$. The classes $\kappa_i(x)$ coming from the conjugation
equation coincide with the classes $\kappa_i^e(\tilde x)$.
\end{proposition}
\begin{proof}
We constructed for $x$ a lift $\tilde{x} \in \HF^{n(1+\alpha)}(X)$ and restriction to the fixed points $X^{\C_2}$ gives us classes
\[
\sum_{i=0}^{n}\kappa_{i}^e(\tilde x)u^{i}a^{n-i} \in \HF^{n(1+ \alpha)}(X^{\C_2})
\]
	
Commutativity of the front left-hand square in the diagram above yields, by construction of $\tilde{x}$, the equality:
\[
\sum_{i=0}^{n}\kappa_{i}^{e}(\tilde{x})u^{i}a^{n-i} = \sum_{i=0}^{n}\kappa_{i}(x)u^{i}a^{n-i}
\]
hence both types of $\kappa$ classes coincide.
\end{proof}

\begin{theorem}\label{thm:conjishomolpur}
If $X$ is a conjugation space of finite type, then it is homologically pure  of finite type.
\end{theorem}

\begin{proof}
Fix a basis of $\HH^{2\ast}(X)$,  finite in any given degree by assumption, and lift each basis element $x_i$, say of degree $2n_i$,  to an element 
\[
\widetilde{x}_i \in \HF^{n_i(1+\alpha)}(X)=[S^{-n_i(1+\alpha)}\wedge X,\HF]^{\C_2} \cong [ X,S^{n_i(1+\alpha)}\wedge\HF]^{\C_2}
\]
Realize each such element by a map $X \rightarrow S^{n_i(1+\alpha)}\wedge \HF$. By additivity we assemble them into a map of equivariant spectra:
		\[
		 X \rightarrow \bigvee_{i \in I} S^{n_i(1+\alpha)} \wedge \HF
		\]
which we extend by linearity into a map of $\HF$-modules:
		\[
		f\colon X \wedge \HF \rightarrow \bigvee_{i \in I} S^{n_i(1+\alpha)} \wedge \HF.
		\]
Let us now check that $f$ is an equivalence, which will prove that $X$ is homologically pure and automatically of finite type. By construction it is an equivalence after forgetting the action. It is thus enough to show 
that it is an equivalence on geometric fixed points by Proposition~\ref{prop:propgeofixedpoints}.(2). We use the computation 
$\Phi^{\C_2}(\HF) = \HH[b]$ done in Proposition~\ref{prop:geofixehf}.		
\[
		\Phi^{\C_2}(f) \colon X^{\C_2} \wedge \HH[b] \rightarrow \bigvee_{i \in I} S^{n_i} \wedge \HH[b].
\]
On homotopy groups, we get a map of $\pi_\ast(H[b])=\F_2[b]$-modules:
\[
\Phi^{\C_2}_\ast(f) \colon \HH_\ast(X^{\C_2})[b] \rightarrow \bigoplus_{i \in I} \HH_\ast(S^{n_i})[b].
\]
After modding out by the maximal ideal $(b)$, we are left with a homomorphism $\HH_\ast(X^{\C_2}) \rightarrow \bigoplus_{i \in I} \HH_\ast(S^{n_i})$
which, by Lemma~\ref{lem:kappanadPhi} is nothing but the map induced in homotopy by summing the maps 
$\kappa_0^e(\tilde{x}_i)$. Proposition~\ref{prop:kappas} shows that these classes are in fact the classes $\kappa_0(x_i)$ and 
$\kappa_0\colon\HH^{2\ast}(X)\rightarrow \HH^\ast(X^{\C_2})$ is an isomorphism by (conjugation) assumption. We conclude by Nakayama's Lemma that 
$\Phi_\ast^{\C_2}$ is an isomorphism. Notice that we needed the finite type assumption to apply Nakayama's Lemma,
or rather a classical consequence of it, see for example \cite[Proposition~2.8]{MR3525784}.
\end{proof}

\section{Properties of conjugation spaces}
\label{sec:properties}

The main result of the previous part can be rephrased as follows. 
\begin{theorem}
\label{thm:rickadef}
Let $X$ be a $\C_2$-space of finite type. Then $X$ is a conjugation space if and only if it is homologically pure  of finite type. 
\end{theorem}

This formulation hints at the fact that being a conjugation space is a property, and not some additional structure on $X$. 
Actually, this (and more) is true. It has been shown that if a $\C_2$-space $X$ admits a conjugation frame, then this structure is unique, 
functorial, and it preserves the multiplicative structure by  \cite[Theorem~3.3, Corollary~3.12]{MR2171799}). It is also 
compatible with the action of the Steenrod operations by work of Franz and Puppe, \cite{MR2198191}.
Although these results are already known, the proofs are quite computational, and we will see how to benefit from the definition 
we advertise here to recover these results in a straightforward, conceptual, and more illuminating fashion.

\subsection{Uniqueness of the conjugation frame and direct consequences}

We offer for completeness a proof of the uniqueness of the conjugation frame. The argument is similar to that of  \cite[Corollary~3.12]{MR2171799}).

\begin{proposition}
Let $X$ be a conjugation space. The conjugation frame $(\kappa, \sigma)$ is  unique.
\end{proposition}

\begin{proof}
Let us assume that $X$ is connected, for simplicity, and consider
$(\kappa,\sigma)$ and $(\kappa',\sigma')$, two conjugation frames on $X$. We have already shown in Proposition~\ref{prop:kappas}, that 
$\kappa = \kappa_0^e= \kappa_0$. Let us show that  $\sigma = \sigma'$. 

Since $\rho \sigma = id_{H^*(X)} = \rho \sigma'$ we have
\[
\sigma(x) = \sigma'(x) + \textrm{ higher terms in } u
\]
for all $x \in H^{2k}(X)$, for all $k \geq 1$. Thus, for degree reasons, we can write $\sigma(x)$ as
\[
\sigma(x) = \sigma'(x) + \sigma'(d_{2k-2})u^2 + \hdots + \sigma'(d_0)u^{2k}
\]
for elements $d_i$ in degree $i$. In particular, if $|x|=0$, then $\sigma(x) = \sigma'(x)$.
The proof follows then by induction on the degree of $x$.
\end{proof}

In particular, for a cohomologically pure space, the frame constructed in Section \ref{sec:pureimpliesconj} is the unique one.
The naturality of this construction implies directly the following.

\begin{proposition}
The conjugation frame is functorial, and multiplicative.
\end{proposition}

\begin{proof}
This in now immediate since the $\HH$-frame is induced by maps of spectra.
\end{proof}

\subsection{Compatibility with the action of the Steenrod algebra}
The idea is to compare the action of the equivariant Steenrod algebra on the cohomology of $X$ with the action of the ordinary
Steenord algebra on $X^u$ and $X^{\C_2}$.
As for ordinary cohomology theories one can consider the set of equivariant cohomology operations in $\HF^\star$-cohomology, which is again a 
Hopf algebroïd. The structure of this algebra has been determined by Hu and Kriz~\cite{HK}. We briefly recall their main theorem
and refer also to recent work of Hill \cite{HillPreprint}. As customary let us adopt the following notation:
\[
\mathcal{A}^\star = [\HF,\HF]^\star
\]
\[
\mathcal{A}_\star = \HF_{\star}(\HF)
\]
The relation between these two objects is via duality, as there is an isomorphism of left $\HF$-modules:
\[
\hom_{\HF_\star}(\mathcal{A}_\star,\HF_\star) \cong \mathcal{A}^\star
\]

\begin{theorem}{\rm (Hu-Kriz, \cite[Theorem~6.41]{HK})}
\label{thm:stralgsteenequivduale}
	The $\HF_\star$-algebra $\mathcal{A}_\star$ admits the following presentation:
	\[
	\mathcal{A}_\star \cong \HF_\star[\xi_{i+1},\tau_i \ | \ i \geq 0]/(\tau_i^2 + a\tau_{i+1} + (a\tau_0 + u)\xi_{i+1})
	\]
	with degrees 
	\[
	|\xi_{i}| = (2^i-1)(1+\alpha) \quad |\tau_i| = (2^i-1)(1+\alpha) +1.
	\]
	Moreover $\mathcal{A}_\star$ has a comodule structure given by the following applications:
	\begin{enumerate}
		\item a right unit $\eta_R\colon \HF_\star \rightarrow \mathcal{A}_\star$ given by $\eta_R(a) = a$ and 
		$\eta_R(u)= a\tau_0 + \sigma{-1}$;
		\item a left unit map $\eta_L\colon \HF_\star\rightarrow \mathcal{A}_\star$ given by the standard inclusion;
		\item a counit $\varepsilon\colon \mathcal{A}_\star \rightarrow  \HF_\star$ uniquely determined by requiring that 
		$\varepsilon \eta_L = Id$ and $\varepsilon(a) = \varepsilon(u) = 0$;
		\item A coproduct $\Delta\colon \mathcal{A}_\star \rightarrow \mathcal{A}_\star \otimes_{\HF_\star} \mathcal{A}_\star$ given by:
		\[
		\Delta(\xi_i) = \sum_{j=0}^i \xi_{i-j}^{2^j} \otimes \xi_j
		\]
		\[
		\Delta(\tau_i) = \sum_{j=0}^i \xi_{i-j}^{2^j} \otimes \tau_j +\tau_i \otimes 1
		\]
		and for $h \in \HF_\star$ $\Delta(h) = \eta_L(h) \otimes 1$.
	\end{enumerate} 
\end{theorem}

Our first, trivial, observation concerns the compatibility of the equivariant Steenrod operations with the  map $\kappa_T$
introduced in Definition~\ref{def:tildekappa} and the section $\sigma$ from Definition~\ref{def:tildesigma}.

\begin{proposition}\label{prop:presdiagcompks}
Let $\theta \in \mathcal{A}^\star$ be a cohomology operation that preserves the diagonal, i.e. of degree $n(\alpha +1)$ for some integer~$n$. 
Then
\[
\kappa_T \theta = \theta \kappa_T \text{ and }  \sigma \theta = \theta \sigma.
\]
\end{proposition}
\begin{proof}
Both morphisms are induced by continuous maps, namely the inclusion of fixed points
$\kappa_T\colon	\HF^{\ast(1+\alpha)}(X) \rightarrow \HF^{\ast(1+\alpha)}(X^{\C_2})$
and the map collapsing the contractible space $E\C_2$ to a point
$\sigma_T\colon 	\HF^{\ast(1+\alpha)}(X) \rightarrow \HF^{\ast(1+\alpha)}(E\C_{2+} \wedge X)$.
We conclude by naturality.
\end{proof}


Our understanding of the action of the Steenrod algebra on the ordinary mod $2$ cohomology comes from the action of the operations that are
dual to the $\xi_1^i$'s on the equivariant cohomology. These operations ``lift" the non-equivariant operation $Sq^{2i}$, in the following sense. For
any equivariant space $X$ we have a commutative diagram:
\[
\xymatrix{
\HF^{\star}(X) \ar[rr]^{(\xi_1^i)^{\vee}} \ar[d]_{\rho} && \HF^{\star + i(1+\alpha)}(X) \ar[d]^{\rho} \\
\HH^{|\star|}(X^u) \ar[rr]^{Sq^{2i}}  && \HH^{|\star| + 2i}(X)
}
\] 
Likewise the operation $\tau_0^\vee$ lifts the Bockstein $Sq^1$. Both statements follow from Hu and Kriz's computations, 
and we refer to the appendix, Section~\ref{sec:forgetaction} for a short explanation of this fact. 
We denote by $Sq$ the total Steenrod square $\sum Sq^\ell$.

\begin{proposition} \label{prop:totalsquarekappa1}{\rm (Franz-Puppe, \cite[Theorem~1.3]{MR2198191})}
Let $X$ be a conjugation space and $x \in \HH^{2*}(X)$. We have an equality 
	\begin{equation*}
	Sq(\kappa_0(x)) = \kappa_0(Sq(x)),
	\end{equation*}
	where $\kappa_0$ is the isomorphism  $\HH^{2*}(X) \cong \HH^*(X^{\C_2})$, part of the conjugation frame.
	Equivalently $\kappa_0(Sq^{2\ell}x) = Sq^\ell\kappa_0(x)$ for any integer~$\ell$.
\end{proposition}

\begin{proof}
Consider the equivariant cohomology class $\tilde{x} \in \HF^{n(1 + \alpha)}(X)$ lifting $x$ as in Lemma~\ref{lem:lift}. Then $(\xi^\ell_1)^\vee(\tilde{x})$ 
has degree $(n+\ell)(1 + \alpha)$. On the one hand $\kappa_T((\xi^\ell_1)^\vee(\tilde{x}))$ decomposes as a sum, see Definition~\ref{def:tildekappa}, 
which reduces modulo $u$ to a single term $a^{n+\ell} \kappa_0^e((\xi^\ell_1)^\vee(\tilde{x})) 
= a^{n+\ell} \kappa_0(\rho[(\xi^\ell_1)^\vee(\tilde{x})])$, where the equality comes from Proposition~\ref{prop:kappas}. By the lifting
property described in the above commutative square we obtain finally, modulo $u$, that
\[
\kappa_T((\xi^\ell_1)^\vee(\tilde{x})) = a^{n+\ell} \kappa_0(Sq^{2\ell}\rho(\tilde{x})) = a^{n+\ell} \kappa_0(Sq^{2\ell}x).
\]
On the other hand we can perform the computation by using first Proposition~\ref{prop:presdiagcompks}, modulo $u$:
\[
\kappa_T((\xi^\ell_1)^\vee(\tilde{x})) = (\xi_1^\ell)^\vee(\kappa_T(\tilde{x})) = 
(\xi_1^\ell)^\vee(\sum_{j=0}^{n}a^{n-j}u^{j}{\kappa}_j^e(\tilde{x})) = (\xi_1^\ell)^\vee(\sum_{j=0}^{n}a^{n-j}u^{j}\kappa_j(x))
\]
where the last equality follows from Lemma~\ref{lem:lift}. Hence
\begin{eqnarray*}
\kappa_T((\xi^\ell_1)^\vee(\tilde{x})) & = &  \sum_{j=0}^{n} a^{n-j} (\xi_1^\ell)^\vee(u^{j}\kappa_j(x)) {\textrm{ by $a$-linearity}}  \\
& = & \sum_{j=0}^n a^{n-j} \sum_{i=0}^\ell (\xi_1^{i})^\vee (u^{j})(\xi_1^{\ell-i})^\vee(\kappa_j(x)) \textrm{ by Cartan formula modulo $u$ }\\
& = & \sum_{j=0}^n a^{n-j}u^{j}(\xi_1^\ell)^\vee(\kappa_j(x)) \textrm{ by Lemma~\ref{lem:valoncoef} }\\
& = & a^n(\xi_1^\ell)^\vee(\kappa_0(x)) \textrm{ modulo $u$. }
\end{eqnarray*}
We infer by Lemma~\ref{lem:actionxiontrivcoef2} that $\kappa_T((\xi^\ell_1)^\vee(\tilde{x})) = a^na^\ell Sq^\ell(\kappa_0(x))$,
modulo $u$. Comparing both terms we conclude that $\kappa_0(Sq^{2\ell}x) = Sq^\ell\kappa_0(x)$.
\end{proof}

\subsection{Identification of the conjugation equation}
Let us now exploit similarly the action of $(\xi_1^\ell\tau_0)^\vee$. 

To begin with, observe that if $X$ is a conjugation space, then   $\HF^{n(1+\alpha)+1}(X) = 0$ for any integer $n$. Indeed, smash the 
cofibration $\C_{2+} \rightarrow S^0 \rightarrow S^\alpha$ with $X$ and consider the associated long exact sequence in $\HF$-cohomology:
\[
\HF^{n(1 + \alpha)}(X) \rightarrow \HH^{2n}(X) \rightarrow \HF^{n(1 + \alpha)+1}(X) \rightarrow \HF^{(n+1)(1 + \alpha)}(X) \rightarrow  \HH^{2(n+1)}(X)
\]
 The two outermost arrows are restriction homomorphisms, see Lemma~\ref{lemma:restriction}. 
As they are isomorphisms by Lemma~\ref{lem:relequivnonequiv}(4),  we get the vanishing result, which can be compared
with a homological version in \cite[Corollary~3.9]{HillPreprint}.
Thus, applying the above operations of degree $\ell(1+\alpha) + 1$ actually kills any class $\tilde{x}$ in $\HF^{n(1+\alpha)}(X)$. 

\begin{proposition}
\label{prop:leskappai}
{\rm (Franz-Puppe, \cite[Theorem~1.1]{MR2198191})}
	Let $X$ be a conjugation space and $x \in \HH^{2n}(X)$. Then $\kappa_{\ell}(x) = Sq^\ell \kappa_0(x)$ for any $\ell \geq 0$.
\end{proposition}

\begin{proof}
The statement is obvious for $\ell = 0$. Let us fix $\ell \geq 0$ and prove the statement for $\ell +1$. 
Our starting point is that $(\xi_1^\ell\tau_0)^\vee(\tilde{x})= 0$, hence that
\[
(\xi_1^\ell\tau_0)^\vee ({\kappa}_T(\tilde{x})) = {\kappa}_T((\xi_1^\ell\tau_0)^\vee(\tilde{x}))= 0.
\]
As above we use $a$-linearity to get  $0= \sum_{j=0}^{n}a^{n-j}(\xi_1^\ell \tau_0)^\vee(u^{j} \kappa_j^e(\tilde{x}))$ and
use the Cartan formula:
\begin{eqnarray*}
0= \sum_{j=0}^n a^{n-j} \left( \sum_{k = 0}^\ell (\xi_1^k\tau_0)^\vee(u^{j})(\xi_1^{\ell-k})^\vee(\kappa_j^e(\tilde{x})) + 
(\xi_1^{\ell-k})^\vee(u^{j})(\xi_1^k\tau_0)^\vee(\kappa_j^e(\tilde{x}))\right. \\
 + \left.\sum_{k=0}^{\ell-1}a (\xi_1^k\tau_0)^\vee(u^{j})(\xi_1^{\ell-1-k}\tau_0)^\vee(\kappa_j^e(\tilde{x})) \right).
\end{eqnarray*}
We apply now Lemma~\ref{lem:actionxiontrivcoef2} to compute the action of the equivariant Steenrod operations on the
equivariant $\kappa_i^e$ classes and identify the latter with $\kappa_i(x)$ by Proposition~\ref{prop:kappas}:
\begin{eqnarray*}
0=  \sum_{j=0}^n \sum_{k = 0}^\ell a^{n-j+ \ell -k}(\xi_1^k\tau_0)^\vee(u^{j})Sq^{\ell-k}(\kappa_j(x)) + 
\sum_{j=0}^n \sum_{k = 0}^\ell a^{n-j+k}(\xi_1^{\ell-k})^\vee(u^{j})Sq^k(\kappa_j(x)) \\
 + \sum_{j=0}^{n}\sum_{k=0}^{\ell-1}a^{n-j+\ell -k} (\xi_1^k\tau_0)^\vee(u^{j})Sq^{\ell-k}(\kappa_j(x)) \\
\end{eqnarray*}
Between the first and third sum all terms cancel two by two but for the terms in the first sum for which $k=\ell$. 
In the second sum, since we compute modulo $u$, only the term for which $j=0 = \ell-k$ survive by
Lemma~\ref{lem:valoncoef} so that:
\begin{eqnarray*}
0 & = & \sum_{j=0}^n a^{n-j}(\xi_1^\ell\tau_0)^\vee(u^{j})\kappa_j(x) + a^na^\ell Sq^{\ell+1}\kappa_0(x).
\end{eqnarray*}
Finally in the first sum, modulo $u$, only the term for which $\ell = j-1$ remains by Lemma~\ref{lem:valoncoef} again
and we are left with $0 = a^{n + \ell} \kappa_{\ell+1}(x) + a^{n + \ell}Sq^{\ell +1}\kappa_0(x)$.
\end{proof}

Observe that, together with Proposition~\ref{prop:totalsquarekappa1}, this shows that for any conjugation space $X$ the following composite operation  
on the fixed point set is independent of $X$ as it coincides with $Sq^\ell$:
\[
\kappa_ \ell \circ \kappa_0 ^{-1}\colon \HH^\ast(X^{\C_2}) \longrightarrow \HH^\ast(X^{\C_2}).
\] 

\subsection{Wrap up with the Steinberg map}
Our final aim is to consider all properties that the cohomology frame of a conjugation space enjoys to express the
property of being a conjugation space in structured algebraic terms. We will use common notation in the study of
modules over the Steenrod algebra, see for example Schwartz's book \cite{MR1282727}. In particular $\mathcal U$
denotes the category of unstable modules over the Steenrod algebra and $\Phi\colon \mathcal U \rightarrow
 \mathcal U$ is the ``doubling'' functor such that $(\Phi M)^{2n} = M^n$ for any unstable module $M$ 
 and the action of $\mathcal A$ is defined by $Sq^{2n} \Phi x = \Phi Sq^n x$.

Let us write $P= \HH^\ast(B\C_{2}) \cong \F[b]$ as in Lannes and Zarati's \cite{MR871217}. 
The $\F$-linear Steinberg map $St\colon M \rightarrow P \otimes M$ is defined by the following formula
for any element $x$ of degree~$n$:
\[
St(x) = \sum_{j=0}^n b^{n-j} \otimes Sq^j x
\]

\begin{corollary}\label{cor:conjequation}
Let $X$ be a conjugation space and $x \in \HH^{2n}(X)$. Then 
\[
r\circ \sigma(x) =  \sum_{j=0}^{n} b^{n-j} Sq^j(\kappa_0(x)).
\]
i.e., the conjugation equation is given by the Steinberg map. \hfill{\qed}
\end{corollary}

In \cite{MR871217}, Lannes and Zarati studied the derived functor of the destabilization map. In particular they show that 
the Steinberg map defines a functor $\text{R}\colon \mathcal{U} \rightarrow \mathcal{U}$, in the sense that, for an unstable module $M$,  
$\text{R}M$ is the $P$-module generated by the image of~$St$. Moreover the Steinberg map is injective, so:

\begin{corollary}\label{cor:derivedstabilization}
Given a conjugation space $X$, there is an isomorphism
\[
\HH^\ast(X_{h\C_2}) \cong \text{\rm R} \HH^\ast(X^{\C_2})
\]
of unstable algebras over the Steenrod algebra, i.e., the Borel cohomology of the equivariant space $X$ is determined 
by the cohomology of the fixed points. \hfill{\qed}
\end{corollary}

This means concretely the following. Let us package the whole structure of a conjugation space into a square:
\[
\xymatrix{
\HH^\ast(X) \ar[d]^{r} \ar@/^1pc/[r]^\sigma \ar[rd]_{\kappa_T} & \HH^\ast(X_{h\C_2})
\ar[d]^r \ar[l] \\
\HH^\ast(X^{\C_2})  & \HH^\ast(X^{\C_2})\otimes \HH^\ast(B\C_{2}) \ar[l]
}
\]
where the horizontal and the vertical arrows are the two kinds of restriction maps introduced in Borel cohomology, except for the
section $\sigma$, and the diagonal map $\kappa_T$ summarizes the conjugation equation. Observe the following consequence 
of the Leray-Hirsch Theorem.
We have an isomorphism $\HH^\ast(X_{h\C_2}) \cong \HH^\ast(X) \otimes \HH^\ast(B\C_2)$ as $\HH^\ast(B\C_2)$-modules, and 
the fact that the section $\sigma$ is  a \emph{ring} map, implies that this is even an isomorphism of algebras. 
Also, the conjugation equation  together with  the fact that the leading term $\kappa_0$ is a isomorphism implies that 
the vertical map $r\colon \HH^\ast(X_{h\C_2})\rightarrow \HH^\ast(X^{\C_2}) \otimes \HH^\ast(B\C_2)$ is injective.
 
The above  diagram coincides then  with the following one:
\[
\xymatrix{
	\Phi\HH^\ast(X^{\C_2}) \ar[d]_{Sq_0}  \ar[rd] _ {St} & \text{R}\HH^\ast(X^{\C_2}) \ar@{^(->}[d]^\rho \ar[l]_-{\rho_1} \\
	\HH^\ast(X^{\C_2})  & \HH^\ast(X^{\C_2})\otimes \HH^\ast(B\C_{2}) \ar[l]
}
\]
where $\rho_1$ is defined in terms of the doubling functor, \cite[Proposition~4.2.6]{MR871217}. We point out that all maps and objects
are functorially determined by $\HH^\ast(X^{\C_2})$ and $Sq_0$ sends $\Phi x$ to $Sq^{|x|} x$.

\appendix

\section{Some computations of cohomology operations}
\label{app:compcohomology}

As it is known  since the classical work of Milnor~\cite{MR0099653}, it is easier to understand the co-action of the dual algebra of cohomology 
operations than the action, essentially because the dual algebra of cohomology  operations is a commutative algebra. In this appendix 
we describe how to switch from the co-action formulas to their duals. A general discussion on stable operations 
in generalized cohomology theories and how to go back and forth between actions and co-.actions can be found in Boardman's contribution to the Handbook of algebraic topology
\cite{MR1361899}, which we use as our main reference  here. Let us also mention Wilson's explicit computations in \cite{Wilson}.
We fix first some notation.

Denote by $\mathcal{A}^\ast =[H,H]^\ast$ the mod $2$ Steenrod algebra and by $\mathcal{A}_\ast$ its dual algebra with respect to $H^\ast = \F$.
Likewise $\mathcal{A}^\star = [\HF, \HF]^\star$ is the equivariant mod $2$ Steenrod algebra  and $\mathcal{A}_\star$ is its $\HF^\star$-dual.
Given an element $x$ in an algebra and an element $\xi$ in the dual, we denote by $\langle x, \xi \rangle = \xi(x)$ the evaluation map.
By \cite{MR0099653}, we know that $\mathcal{A}_\ast$ is isomorphic to a polynomial algebra $\F_2[\zeta_i, i \geq 1]$ on classes $\zeta_i$
of degree $2^i - 1$.

The equivariant Steenrod algebra, although  more complicated, is still a commutative algebra, as we saw from the computations of Hu and Kriz
in Theorem~\ref{thm:stralgsteenequivduale}. 
It is generated by the elements $\xi_{j}$ and $\tau_j$ (where it is sometimes handy to set $\xi_0 =1$).

\begin{definition}\label{def:monomialbasis}
The set $\mathcal{MB} = \{ \xi_j^\ell \tau_i,  \xi_j^\ell \, \mid \, j \geq 0, i \geq 0 \}$ forms an  $\HF_ \star$-basis of the algebra $\mathcal{A}_\star$ 
which refer to as the \emph{monomial basis}. The dual elements in $\mathcal A^\star$ will be denoted by $(\xi_j^\ell \tau_i)^\vee$ and $(\xi_j^\ell)^\vee$ respectively.
\end{definition}

\subsection{Forgetting the action}\label{sec:forgetaction}
Evaluating along the structural map $\rho$ in Mackey functors over $\C_2$, we get a map 
\[
\mathcal{A}_\star = \HF_{\star}[\tau_i,\xi_{i+1}]/(\tau_i^2 + a \tau_{i+1} + (a \tau_0 + u)\xi_{i+1}) \longrightarrow \F_2[\zeta_i] = \mathcal{A}_\ast.
\]
The structure of the coefficients $\HF_\star$ recorded in Proposition~\ref{prop:strMAckHF} implies that the class $a$ restricts to zero 
and the class $u$ restricts to the 
unique non-trivial class in $[(S^{-1+\alpha})^u, \HF^u]=[S^0,\HH] \cong \F$. In particular the map above factors through 
$(\HF_\star)_{\C_2}[\tau_i, \xi_{i+1}]/(\tau_i^2 = u \xi_{i+1})$, which we can identify by the above discussion with $\F[\tau_i]$. 
For degree reasons, $\tau_1$ can only map to $0$ or $\zeta_1$, and the computation of the action of Hu and Kritz~\cite[Lemma~6.27]{HK} on the space 
$B\Z_2'$ shows that the action is not trivial. 
As $Sq^1$ is the dual of $\zeta_1$ this shows that the dual $\tau_1^\vee$ lifts indeed $Sq^1$; then $\xi_1^\vee$ lifts the dual of the image 
of $\xi_1 = \tau_1^2$ which is $\zeta_1^2$, and the dual of this last map is indeed $Sq^2$. 
The same argument shows, since we have an algebra map, that $(\xi_1^j\tau_1^\varepsilon)^\vee$ lifts the dual of the image of 
$\tau_1^{2j} \tau_1^\varepsilon$ where $\varepsilon = 0$ or $1$. This is $\zeta_1^{2j + \varepsilon}$, whose dual in turn is $Sq^{2j+\varepsilon}$.  

\subsection{Cartan formulas and the action on the coefficients}
 The Cartan formula describes the action of a cohomology operation $\theta$ on the cup product of classes. A general explanation of the Cartan formula 
 can be found in \cite[Section~12]{MR1361899}. Since the right unit $\eta_R$ in $\mathcal{A}_\star$ encodes the action on the coefficients, 
 and is not the identity, contrary to what happens with the non-equivariant Steenrod algebra, $\mathcal{A}^\star$ does act on the coefficients. 
 To understand this action it is enough to compute it on the basis elements $(\xi_1^\ell \tau_0)^\vee$ and $(\xi_1^\ell)^\vee$. 
 In general, for such a basis element $(\theta)^\vee$ the Cartan formula reads as follows:
\[
\nabla((\theta)^\vee) = \sum_{(h,x_\alpha,x_\beta)}h x_ \alpha^\vee \otimes x_ \beta^\vee,
\]
where the sum is taken over the triples $(h,x_\alpha,x_\beta) \in \HF_\star \times \mathcal{MB} \times \mathcal{MB}$ such that 
$\langle \theta^\vee, x_ax_b \rangle = \eta_R(h)$.

Because of the relation $\tau_0 ^2 = a\tau_1 + (a\tau_0 + u)\xi_1$ and  $\eta_R(a)= a$ $\eta(u) = a\tau_0 + u$ we have:
\[
\nabla((\xi_1^\ell)^\vee) = \sum_{j=0}^\ell (\xi_1^j)^\vee \otimes (\xi_1^{(\ell -j)})^\vee + \sum_{j=0}^{\ell-1} u (\xi_1^j\tau_0)^\vee \otimes (\xi_1^{(\ell-1-j)}\tau_0)^\vee
\]
and
\begin{eqnarray*}
\nabla((\xi_1^\ell\tau_0)^\vee) & =&  \sum_{j=0}^\ell (\xi_1^j\tau_0)^\vee \otimes (\xi_1^{(\ell -j)})^\vee +   (\xi_1^{(\ell-j)})^\vee \otimes (\xi_1^j\tau_0)^\vee  \\
& &+ \sum_{j=0}^{\ell-1}a (\xi_1^j\tau_0)^\vee \otimes (\xi_1^{(\ell-1-j)}\tau_0)^\vee .
\end{eqnarray*}

As explained in \cite[Lemma 12.6]{MR1361899}, the coaction of $\A_ \star$ on $\HF_\star$ is encoded in the right unit $\eta_R$. 
From $\eta_R(a)=a$ we get that all elements of $\A^\star$ are $\F_2[a]$-morphisms. From $\eta_R(u)= a\tau_0 + u$ we get $\tau_0^\vee(u)=a$ and  $(\xi_1^\ell)^\vee(u)=0$. We apply next the Cartan formula to compute
\begin{eqnarray*}
(\xi^\ell_1)^\vee(u^{k+1}) & = & \sum_{j=0}^\ell (\xi_1^j)^\vee(u^{k}) (\xi_1^{(\ell -j)})^\vee(u) + 
\sum_{j=0}^{\ell-1} u (\xi_1^j\tau_0)^\vee(u^{k})  (\xi_1^{(\ell-1-j)}\tau_0)^\vee(u)\\
& = & u(\xi_1^\ell)^\vee(u^k)  + au(\xi_1^{\ell-1}\tau_0)^\vee(u^{k}),
\end{eqnarray*}
and analogously
\begin{eqnarray*}
(\xi_1^\ell\tau_0)^\vee(u^{k+1}) & = & \sum_{j=0}^\ell (\xi_1^j\tau_0)^\vee(u^{k})  (\xi_1^{(\ell -j)})^\vee(u) +   (\xi_1^{(\ell-j)})^\vee(u^{k}) (\xi_1^j\tau_0)^\vee(u)  \\
& &+ \sum_{j=0}^{\ell-1}a (\xi_1^j\tau_0)^\vee(u^{k}) (\xi_1^{(\ell-1-j)}\tau_0)^\vee(u) \\
& = & u (\xi_1^\ell\tau_0)^\vee(u^{k})  + a(\xi_1^\ell)^\vee(u^{k}) + a^2 (\xi_1^{\ell-1}\tau_0)^\vee(u^{k})
\end{eqnarray*}

These formulas suffice to derive an exact computation by induction, but for our present purposes we only need the computation modulo $u$, 
which is now almost immediate.

\begin{lemma}\label{lem:valoncoef}
For any $k \geq 1$ and $\ell \geq 0$
 \[
 (\xi_1^{\ell+1})^\vee(u^{k}) = 0 \textrm{ mod } u
 \]
 \[
 (\xi_1^{\ell}\tau_0)^\vee(u^{k}) = \left\{
 	\begin{array}{rl}
 	0 \text{ mod } u& \textrm{ if } \ell \neq k-1 \\
 	 a^{2k-1} & \textrm{ if } \ell = k-1
 	\end{array}
 	\right.
 \]
\end{lemma}

\subsection{Action on trivial spaces}

We proved in Proposition~\ref{prop:cohoesptriv} that the equivariant cohomology of a spectrum $Y$ with trivial action is determined by the ordinary cohomology. 
Our aim here is to understand to what extent the action of the stable cohomology operations on $\HF^\star (\iota Y)$
is determined by the action of the ordinary Steenrod algebra on $\HH^\ast Y$.
Recall, e.g. from \cite{MR1361899}, that the action of the Steenrod algebra is encoded in the coaction by the dual Steenrod algebra. By 
\cite[Definition~C.3]{Rickatmm} it decomposes as:
\[
\HF^\ast \otimes \HH^\ast Y \stackrel{1 \otimes \lambda}{\longrightarrow} \HF^\ast \otimes \A_\ast \otimes \HH^\ast Y 
\stackrel{\Psi \otimes 1}{\longrightarrow} \A_\star \otimes \HH^\ast Y \cong \A_\star \otimes_{\HH^\ast} \HH^\ast \otimes \HH^\ast Y
\]
where $\lambda$ is the non-equivariant coaction map, and $\Psi$ is the $\HF^\star$-module map constructed as follows.
Let $\varepsilon\colon \iota \HH \rightarrow \HF$ be the left adjoint to the identity map $\HH \rightarrow (\HF)^{\C_2}$.
Given a non-equivariant spectrum $Y \in \Sp$, Proposition~\ref{prop:cohoesptriv} says that we have an equivariant equivalence 
$\HF \wedge_{\iota H} \iota \HH \wedge \iota Y \stackrel{\sim}{\longrightarrow} \HF \wedge \iota Y$.
The case $Y= \iota \HH$ gives us a map:
\[
\Psi\colon \HF \wedge_{\iota H} \iota(\HH \wedge \HH) \rightarrow \HF \wedge \iota H \stackrel{Id \wedge \varepsilon}{\longrightarrow}
\HF \wedge \HF
\]
which in homotopy induces:
\[
\Psi\colon \HF^\star \otimes \A_\ast \rightarrow \A_\star
\]
and encodes precisely the way $\A_\star$ co-acts on the cohomology of a trivial spectrum through the coaction of the non-equivariant 
$\A_\ast$. By linearity it is thus enough to compute the restriction $\psi\colon \A_\ast \rightarrow \A_\star$, which has been done
explicitly by Ricka~\cite{Rickatmm}. The images of Milnor's polynomial generators suffice to describe $\psi$.

\begin{proposition}\label{prop:computpsi}\cite[Theorem~D.3]{Rickatmm}
	The algebra map $\psi\colon \mathcal{A}_\ast \rightarrow \mathcal{A}_\star$ is determined by the formula:
	\[
	\psi(\zeta_n) = a^{2^n-1}\xi_n + \sum_{i=1}^n a^{2^n-2^i}\eta_r(u)^{2^{i-1}-1}\xi_{n-i}^{2^{i-i}}\tau_{n-1}.
	\]
	
\end{proposition}

We want to compute the elements $\langle (\xi_1^i\tau_0)^\vee, \psi(\theta)\rangle,\langle (\xi_1^i)^\vee, \psi(\theta)\rangle  \in \HF^\star$
for arbitrary $\theta \in \mathcal{A}_\star$. 
An inspection of the formula in Proposition~\ref{prop:computpsi} shows that the monomials $\xi_1^i$ or $\xi_1^i\tau_0$ can appear in the expansion 
of $\psi(\zeta_n^k)$   if and only if $n =1$, so it is enough to compute the elements
\[
\begin{array}{cc}
C^k_j &= \langle (\xi_1^j)^\vee, \psi(\zeta_1^k)\rangle  \in \HF^\star, \\
D^k_j&=\langle (\xi_1^j)^\vee, \psi(\zeta_1^k)\rangle  \in \HF^\star.
\end{array} 
\]
In particular in the target of the map $\psi$ we may work modding out the elements $\xi_k$ for $k \geq 2$ and $\tau_k$ for $k \geq 1$.
Let us set
\[
\overline{\mathcal{A}}_\star = \mathcal{A}_\star/(\tau_k,\xi_{k+1}; k \geq 1) \cong \HF_\star[\tau_0,\xi_1]/(\tau_0^2 + a\xi_1\tau_0 + u\xi_1).
\]
Proposition~\ref{prop:computpsi} implies that the induced map 
$\mathcal{A}_\ast \rightarrow \overline{\mathcal{A}}_\star$ factors through the quotient $\mathcal{A}_\ast/(\zeta_k, k \geq 2)$, 
more precisely we have a commutative diagram:
\[
\xymatrix{
	\F_2[\zeta_1] \ar[r] \ar@{=}[dr] & \mathcal{A}_\ast \ar[d] \ar[r]^{\psi} & \mathcal{A}_\star \ar[d] \\
	& \mathcal{A}_\ast/(\zeta_k, k \geq 2) \ar[r]^-{\overline{\psi}} & \overline{\mathcal{A}}_\star
}
\]
By construction the map $\overline{\psi}\colon \mathcal{A}_\ast/(\overline{\xi}_k, k \geq 2) \rightarrow \overline{\mathcal{A}}_\star$ is 
again an algebra map.
An element in $\overline{\mathcal{A}}_\star$ can be written in a unique way as a sum of two polynomials, one in $\xi_1$, the other one
in $\xi_1$ times $\tau_0$. Let us thus define two sequences of polynomials $P_n$ and $Q_n$ in $\xi_1$  by the rule:
\[
\overline{\psi}(\zeta_1^n) = P_n + Q_n \tau_0.
\]
According to Proposition~\ref{prop:computpsi},  $\overline{\psi}(\zeta_1)= a\xi_1 + \tau_0$, so $P_1 = a\xi_1$ and $Q_1=1$.
As $\overline{\psi}$ is an algebra map, $\overline{\psi}(\zeta_1^{n+1})=\overline{\psi}(\zeta_1^{n})\overline{\psi}(\zeta_1)$ 
and we get an inductive formula for any $n$:
\[
P_{n+1} = a\xi_1 P_{n} + u\xi_1 Q_n \ \text{ and } \ Q_{n+1} = P_n.
\]
More compactly, $\psi(\zeta_1^{n+1}) = P_{n+1} + P_n \tau_0$, where the polynomials $P_n$ are determined
by $P_0=1$, $P_1 = a\xi_1$, and inductively, for any $n \geq 0$,
\[
P_{n+2} = a\xi_1 P_{n+1} + u\xi_1 P_n.
\]
By construction we have the equality $\langle (\xi_1^{i})^\vee, \psi(\zeta_1^{k}) \rangle = \langle (\xi_1^{i})^\vee, P_{k} \rangle$
in $\HF^\star$ and
\[
D^k_j= \langle (\xi_1^j\tau_0)^\vee, Q_k\tau_0 \rangle = \langle (\xi_1^j\tau_0)^\vee, Q_k\tau_0 \rangle = \langle (\xi_1^j)^\vee, P_{k-1} \rangle = C^{k-1}_j.
\]
Moreover, the inductive relation $C^{k+1}_{j+1} = a C^k_j + uC_j^{k-1}$ follows from the ones established for the $P_n$'s.
An easy induction gives then an explicit description.

\begin{lemma}\label{lem:calculxicontreP}
	For any $i \geq 0$ and $k \geq 0, \quad C_i^k = \langle (\xi_1^{i})^\vee, \psi(\zeta_1^{k}) \rangle = \binom{i}{k-i}a^{2i-k}u^{(k-i)}$.
	\end{lemma}
	
Since $\zeta_1^k$ is the dual element to the Steenrod square $Sq^k$, the above lemma and the way $\psi$ encodes 
the action of Steenrod algebra yields for any $y \in \HH^n(Y)$: 
	\[
	(\xi_1^i)^\vee(y) = \sum_{k=i}^{2i}\binom{i}{k-i}a^{2i-k}u^{(k-i)}Sq^{k}(y)
	\]
	\[
	(\xi_1^i\tau_0)^\vee(y) = \sum_{k=i+1}^{2i+1}\binom{i}{k-1-i}a^{2i-k-1}u^{(k-i-1)}Sq^{k}(y)
	\]
The change of variables $j=k-i$ in the first sum and $j= k-i-1$ in the second provides us finally with the formulas we were looking for.
 The action of the cohomology operations $(\xi_1^i)^{\vee}$ and $(\xi_1^i\tau_0)^\vee$ on the equivariant cohomology of a trivial $\C_2$-spectrum $Y$ are 
entirely determined by the following formulas.

\begin{lemma}\label{lem:actionxiontrivcoef2}
Let $Y$ be a trivial $\C_2$-spectrum and $y \in \HH^n(Y)$. Then
\[
	(\xi_1^i)^{\vee}(y) = \sum_{j=0}^i \binom{i}{j} Sq^{i+j}(y)u^{j}a^{i-j} \ \text{and} \ 
	(\xi_1^i\tau_0)^{\vee}(y) = \sum_{j=0}^i \binom{i}{j} Sq^{i+j+1}(y)u^{j}a^{i-j}.
\]
	\end{lemma}

\bibliographystyle{plain}\label{biblography}

\end{document}